\documentclass[a4paper,reqno]{amsart}

\pdfoutput=1

\usepackage{microtype}
\usepackage[utf8]{inputenc}

\usepackage{amsmath,amstext,amssymb,mathrsfs,amscd,amsthm,indentfirst,bbm}
\usepackage{amsfonts}
\usepackage{stmaryrd}

\usepackage{todonotes}

\usepackage[style=ieee-alphabetic, url=false, backref=true, doi=false, isbn=false, eprint=false]{biblatex}
\DeclareFieldFormat{url}{\textsc{url}\addcolon\space\url{#1}}

\usepackage{float}
\usepackage{adjustbox}
\usepackage{booktabs}
\usepackage{verbatim}
\usepackage{enumerate}
\usepackage{graphicx}
\usepackage{textcomp}
\usepackage{tikz,tikz-cd}
\usetikzlibrary{patterns}
\usepackage{ifthen}
\usepackage{url}

\usepackage{xcolor}
\definecolor{lightgreen}{RGB}{45,137,105}
\definecolor{purple}{RGB}{155,26,103}
\usepackage{hyperref}
\hypersetup{bookmarksdepth=3}
\PassOptionsToPackage{hyphens}{url}\usepackage{hyperref}
\hypersetup{
    colorlinks  = true,
    linkcolor   = purple,
    citecolor   = lightgreen,
    urlcolor    = [rgb]{0,.40,.80}
}
\usepackage[nameinlink,capitalize]{cleveref}
\usepackage{mathtools}
\usepackage{mleftright}
\usepackage{enumitem}

\usepackage[new]{old-arrows}

\usepackage{libertine}
\makeatletter
\renewcommand\th@plain{\slshape}
\makeatother

\calclayout

\newcommand{\bC}{\mathbb C}
\newcommand{\bE}{\mathbb E}

\newcommand{\bN}{\mathbb N}
\newcommand{\bP}{\mathbb P}

\newcommand{\bS}{\mathbb S}

\newcommand{\bZ}{\mathbb Z}

\newcommand{\cC}{\mathcal C}

\newcommand{\cH}{\mathcal H}

\newcommand{\cL}{\mathcal L}

\newcommand{\cO}{\mathcal O}
\newcommand{\cP}{\mathcal P}
\newcommand{\cQ}{\mathcal Q}

\newcommand{\lra}{\longrightarrow}

\newcommand{\Gammahol}{\Gamma_{\mathrm{hol}}}
\newcommand{\Gammacon}{\Gamma_{\cC^0}}

\newcommand{\ad}{\mathrm{ad}}

\newcommand{\Pic}{\mathrm{Pic}}

\newcommand{\Hol}{\mathrm{Hol}}
\newcommand{\Rat}{\mathrm{Rat}}
\newcommand{\Map}{\mathrm{Map}}
\newcommand{\Hom}{\mathrm{Hom}}
\newcommand{\NS}{\mathrm{NS}}

\newcommand{\Grass}{\mathrm{Gr}}
\newcommand{\Flag}{\mathrm{Fl}}
\newcommand{\UConf}{\mathrm{UConf}}

\newcommand{\Conf}{\mathrm{Conf}}
\newcommand{\Mono}{\mathrm{Mono}}

\newcommand{\catSp}{\mathsf{Sp}}
\newcommand{\catJ}{\mathsf{J}}
\newcommand{\catS}{\mathsf{S}}

\newcommand{\opp}{\mathrm{op}}

\newcommand{\catEpi}{\mathsf{Epi}}
\newcommand{\catFun}{\mathsf{Fun}}
\newcommand{\catMod}{\mathsf{Mod}}

\DeclareMathOperator{\Ind}{Ind}

\DeclareMathOperator{\rank}{rk}
\DeclareMathOperator{\image}{im}

\DeclareMathOperator*{\colim}{colim}

\NewDocumentCommand{\set}{somm}{%
   \IfNoValueTF{#2}
    {\IfBooleanTF{#1}{\{#3 \mid #4\}}{\mleft\{ #3 \mathrel{}\middle\vert\mathrel{} #4 \mright\}}}
    {\mathopen{#2\{}#3 \mathrel{}#2\vert\mathrel{} #4\mathclose{#2\}}}%
  }

\newtheorem{theorem}{Theorem}[section]
\newtheorem*{theorem*}{Theorem}
\newtheorem{theorem-in-progress}[theorem]{Theorem-in-progress}

\newtheorem*{conjecture*}{Conjecture}

\newtheorem{definition}[theorem]{Definition}
\newtheorem{lemma}[theorem]{Lemma}
\newtheorem{remark}[theorem]{Remark}
\newtheorem{proposition}[theorem]{Proposition}
\newtheorem{example}[theorem]{Example}

\newtheorem{corollary}[theorem]{Corollary}

\newtheorem{maintheorem}{Theorem}

\newtheorem*{theorem:unitary-cccm}{Theorem~\ref{theorem:unitary-ccmm}}
\newtheorem*{proposition:top-derivative}{Proposition~\ref{proposition:top-derivative}}

\title{Weiss derivatives of holomorphic maps}
\author{Alexis Aumonier}
\address{Department of Mathematics, Stockholms universitet, 106 91 Stockholm, Sweden}
\email{alexis.aumonier@math.su.se}

\begin{document}

\begin{abstract}
We propose an orthogonal approach to the stable homotopy type of spaces of holomorphic maps to projective space. We study the Weiss towers of the unitary functors of holomorphic and continuous maps to $\mathbb{P}(V)$, and show that the former is polynomial and completely compute the latter. As an application we give a new proof of a stable splitting of Cohen--Cohen--Mann--Milgram.
\end{abstract}

\maketitle

\setcounter{tocdepth}{1}
\tableofcontents

\section{Introduction}

Let $X$ be a smooth projective complex variety. The space of holomorphic maps to projective space
\[
    \Hol(X, \bP^m) \subset \Map(X, \bP^m)
\]
naturally sits as a subspace of the space of continuous maps. The connectivity of this inclusion was first studied by Segal in his seminal work \cite{Segal} in the case of $X$ a Riemann surface of genus~$g$. His main contribution shows that, when restricted to the components of maps of degree $d$, the inclusion induces an isomorphism in homology in the range of degrees $* \leq (d-2g)(2m-1)$. A similar result holds for $X=\bP^n$ by \cite{Mostovoy}, or more generally any $X$ as above by \cite{Aumonier}. This fundamentally splits the problem of computing the homology of the space of holomorphic maps into two distinct parts: the \emph{stable range}, where the homology agrees with that of the continuous mapping space, and the \emph{unstable range}. The theory of models in rational homotopy gives a very good answer if one wishes to compute the stable homology with rational coefficients. In fact, it is in this setting that one observes a phenomenon of homological stability for maps of increasing degrees, giving its name to the stable range. The other part, the unstable range, is much less accessible. It seems to highly depend on the algebraic geometry of $X$ and very few full computations have been attempted \cite{CCMM,,Havlicek,,Totaro,,KallelMilgram}. The present paper is concerned with this unstable range.

Our contribution is not to directly give new computations of some unstable homology groups, but rather to suggest a way of organising the computation. Our approach exploits the functoriality
\[
    \text{complex vector space } V \longmapsto \Hol(X, \bP(V)),
\]
which can be studied within the framework of orthogonal calculus developed by Weiss \cite{Weiss}. We have collected in \cref{subsection:aide-memoire} the main points of this theory, but in the rest of this introduction we shall also spell out explicitly the consequences of our results.

Let $n$ be the dimension of $X$, and for any $m \geq n$ let $h \in H^2(\bP^m;\bZ) \cong \bZ$ be the positive generator. We say that a map $f \colon X \to \bP^m$ has degree $\deg(f) = f^*(h) \in H^2(X;\bZ)$. The mapping space $\Map(X,\bP^m)$ splits into connected components indexed by the degrees, and we write
\begin{equation}\label{equation:inclusion}
    \Hol_\alpha(X,\bP^m) \subset \Map_\alpha(X,\bP^m)
\end{equation}
for the subspace\footnote{Not necessarily connected nor non-empty, though it will be if $\alpha$ is ample enough.} of holomorphic maps of degree $\alpha \in H^2(X;\bZ)$. We showed in \cite{Aumonier} that this inclusion induces an isomorphism in homology in the range of degrees $* < (2m-2n+1)d(\alpha) + 2(m-n-1)$, where $d(\alpha)$ measures the ampleness of the class $\alpha$. We will recall a few properties of the number $d(\alpha)$ in \cref{subsection:conclusion}, but let us only say now that for $X = \bP^n$ and $\deg\alpha = d \in \bN$ it is $d$, and for $X = \Sigma_g$ and $\deg\alpha = d \in \bN$ it is $d-2g$. The main result of this article is:
\begin{maintheorem}\label{maintheorem:holomorphic}
Let $X$ be a connected smooth projective complex variety. Let $\alpha \in H^2(X;\bZ)$ and suppose that there exists a very ample line bundle $\cL$ on $X$ with $c_1(\cL) = \alpha$ and such that $c_1(\cL \otimes K_X^\vee)$ is ample, where $K_X^\vee$ is the dual of the canonical bundle of $X$. Let $N = \dim H^0(X,\cL)$ and $n = \dim X$. Then the unitary functor
\[
    V \longmapsto \Sigma^\infty_+\Hol_\alpha(X,\bP(\bC^{n+1} \oplus V))
\]
is $N$-polynomial\footnote{The dimension $N$ is independent of the chosen $\cL$ by the Kodaira vanishing theorem.}. In particular, there is a tower of spectra functorial in $X$ and $V$
\[
    \Sigma^\infty_+\Hol_\alpha(X,\bP(\bC^{n+1} \oplus V)) = T_N(V) \lra T_{N-1}(V) \lra \cdots \lra T_1(V) \lra T_0
\]
starting with $T_0 \simeq \Sigma^\infty_+ \left( K(H^1(X;\bZ),1) \times \bC\bP^\infty\right)$, and such that for each $1 \leq k \leq N$ there exists a spectrum $\Theta_k$ with an action of the unitary group $U(k)$ and a fibre sequence
\[
    \left(\Theta_k \otimes \bS^{\bC^k \otimes V}\right)_{hU(k)} \lra T_k(V) \lra T_{k-1}(V).
\]
Furthermore, for $1 \leq k \leq d(\alpha)$ there is an explicit construction of $T_k(V)$ and 
\[
    \Theta_k \simeq \left(\Ind_{U(1) \wr \Sigma_k}^{U(k)} \Sigma^{k(2n+1)}  \Conf_k(X, S(\cL))^{-kTX}\right) \sslash \Map(X,U(1))
\]
where $\Conf_k(X, S(\cL))$ is the configuration space of $k$ points in $X$ with labels in the sphere bundle $S(\cL)$. In that formula, the unitary group $U(1)$ acts on $S(\cL)$ and the symmetric group $\Sigma_k$ acts by permutations of factors and configurations. (See \cref{theorem:final-continuous-answer} for more details.)
\end{maintheorem}
Our suggested interpretation of this theorem is as follows. To understand the stable homotopy type of $\Hol_\alpha(X,\bP^m)$ it ``suffices" to understand \emph{finitely} many spectra $\Theta_1,\dotsc,\Theta_N$ and finitely many extensions problems. Furthermore, the first $d(\alpha)$ spectra and extensions are already determined. Of course, this will not be easy in general. In particular notice that for $X = \bP^n$ and $\deg \alpha = d$, the number $N = \binom{n+d}{d} = \frac{d^n}{n!} + o(d^n)$ grows much faster than $d(\alpha) = d$ if $n \geq 2$. We will see shortly what more can be said when $n = 1$. Despite this obstacle, our theorem pinpoints precise questions to ask about the unstable range, e.g. what is $\Theta_{d(\alpha)+1}$?

\bigskip

The polynomiality of the unitary functor $\Sigma^\infty_+\Hol_\alpha(X,\bP(\bC^{n+1} \oplus V))$ in \cref{maintheorem:holomorphic} is established by studying a filtration on holomorphic maps, essentially given by the dimension of the projective span of their images. Although our proof is very concrete and theoretically directly applicable to compute the spectra $\Theta_k$ in examples, as we shall see in \cref{section:case-study}, the computations quickly grow in size. Thus, to obtain a formula for $T_k(V)$ and $\Theta_k$ with $1 \leq k \leq d(\alpha)$, we study the analogous unitary functor of continuous maps. Being entirely in the realm of homotopy theory, we are are able to fully understand its Weiss tower.
\begin{maintheorem}\label{maintheorem:continuous}
Let $X$ be a finite CW complex of finite dimension $d$. Let $\cL$ be a complex line bundle on $X$ with first Chern class $c_1(\cL) = \alpha \in H^2(X;\bZ)$. Let $M > \frac{d}{2}$. Then the unitary functor
\[
    V \longmapsto \Sigma^\infty_+ \Map_\alpha(X, \bP(\bC^{M} \oplus V))
\]
has a converging Weiss tower:
\[
    \Sigma^\infty_+ \Map_\alpha(X, \bP(\bC^{M} \oplus V)) \simeq \lim \left( \cdots \lra \widetilde{T}_k(V) \lra \widetilde{T}_{k-1}(V) \lra \cdots \lra \widetilde{T}_1(V) \lra \widetilde{T}_0 \right)
\]
with $\widetilde{T}_0 \simeq T_0 \simeq \Sigma^\infty_+ \left( K(H^1(X;\bZ),1) \times \bC\bP^\infty\right)$. In particular, for each $k \geq 1$ there exists a spectrum $\widetilde{\Theta}_k$ with an action of the unitary group $U(k)$ and a fibre sequence
\[
    \left(\widetilde{\Theta}_k \otimes \bS^{\bC^k \otimes V}\right)_{hU(k)} \lra \widetilde{T}_k(V) \lra \widetilde{T}_{k-1}(V).
\]
Furthermore, for all $k \geq 1$ there is an explicit construction of $\widetilde{T}_k(V)$ and 
\[
    \widetilde{\Theta}_k \simeq \left(\Ind_{U(1) \wr \Sigma_k}^{U(k)} \Sigma^{k(2M-1)}  \Conf_k(X, S(\cL))^{-kTX}\right) \sslash \Map(X,U(1)).
\]
\end{maintheorem}
The precise version of the theorem appears as \cref{theorem:final-continuous-answer} in the text. Although it is stated for a finite dimensional CW complex $X$, we only use it in this paper for $X$ a smooth projective variety as a tool to compute the $\Theta_k$ in the range $1 \leq 1 \leq d(\alpha)$. Indeed, the fact that $\Theta_k \simeq \widetilde{\Theta}_k$ in that range follows from the homological connectivity of the inclusion~\eqref{equation:inclusion} and standard facts in unitary calculus. (We will see this in \cref{subsection:conclusion}.)

\subsection{Applications}
As mentioned above, when the variety $X$ is one-dimensional our results give most of the Weiss tower. In fact, for pointed holomorphic maps from $\bP^1$ we indeed compute the full tower. We exploit this observation to give a new conceptual proof of a stable splitting of Cohen--Cohen--Mann--Milgram \cite{CCMM}.
\begin{theorem:unitary-cccm}
The space of pointed holomorphic maps from $\bP^1$ stably splits:
\[
    \Sigma^\infty_+ \Hol_d^*(\bP^1, \bP(\bC^2 \oplus V)) \simeq \bigoplus_{i=0}^d \Sigma^{i\dim V}\left(\Conf_i(\bC)_+ \wedge (S^1)^{\wedge i}\right)_{\Sigma_i}.
\]
\end{theorem:unitary-cccm}
In our language, we say that the functor $\Sigma^\infty_+ \Hol_d^*(\bP^1, \bP(\bC^2 \oplus V))$ is $d$-polynomial and that its Weiss tower splits. We do not expect to observe such splittings in other examples, though it would be very interesting. For unpointed maps from $\bP^1$ of degree $d$, only the top derivative $\Theta_{d+1}$ is not given by our main theorem. We are nonetheless able to compute it. In fact, and perhaps surprisingly, we can give a formula for the top derivative $\Theta_N$ in general.
\begin{proposition:top-derivative}
Let $X$ be a connected smooth projective complex variety with $H^1(X;\bZ) = 0$, and let $\alpha \in H^2(X;\bZ)$, $\cL$ and $N$ be as in \cref{maintheorem:holomorphic}. Let $\iota \colon X \hookrightarrow \bP(\bC^N)$ be the canonical embedding associated to the very ample line bundle $\cL$.
Then the $N$th derivative of the $N$-polynomial unitary functor
\[
    V \longmapsto \Sigma^\infty_+\Hol_\alpha(X,\bP(\bC^N \oplus V))
\]
is given by
\[
    \Theta_N \simeq \bS^{\ad_N} \otimes \Sigma^\infty_+PU(N) \otimes \Sigma^\infty\Sigma^\mathrm{un}|\Grass(X,\cL)|
\]
where $|\Grass(X,\cL)|$ is the geometric realisation of the topological poset $\Grass(X,\cL)$ of non-empty linear subspaces $\emptyset \neq P \subset \bP(\bC^N) - \iota X$ ordered by inclusion, $\Sigma^\mathrm{un}$ denotes the unreduced suspension, and $\ad_N$ is the adjoint representation of $U(N)$.
\end{proposition:top-derivative}
The assumption on the first cohomology of $X$ is purely to simplify the poset. In general, one needs to also vary $\cL$ in the Picard variety $\Pic^\alpha(X)$, see \cref{subsection:top-derivative} for more details. Spaces of linear projective subspaces not intersecting a given subset, such as the poset in our result above, appear repeatedly throughout our arguments. We think it would be worthwhile to study their topology.

\bigskip

Our proposal in this article is that orthogonal calculus divides the unstable homology of spaces of holomorphic maps, and lets us ask about the pieces. Notably, our approach lends itself to concrete computations, and we have thus collected a few of them in the last \cref{section:case-study}. Our motivation there is mainly to explain how the theory works in concrete examples.

\subsection*{Outline}
We recall the basics of unitary calculus in \cref{section:unitary-calculus}. In \cref{section:polynomiality} we prove the polynomiality statement of \cref{maintheorem:holomorphic}. The computation of the derivatives is deduced in \cref{subsection:conclusion}, as it follows from \cref{maintheorem:continuous} proven in \cref{section:continuous}. We apply our results to obtain the stable splitting of \cref{theorem:unitary-ccmm} in \cref{section:splitting}. Finally, in \cref{section:case-study} we showcase some computations in the theory.

\subsection*{Conventions}
We write $\catS$ and $\catSp$ for the $\infty$-categories of spaces and spectra, but we are lax about which model the reader prefers: we will only use formal manipulations. Given a space $B$, we will often use the category of parameterised spaces over $B$ denoted $\catS^B$: we will implicitly use the (un)straightening equivalence and see an object $Y \in \catS^B$ as either a functor $Y \colon B \to \catS$ or a space $Y \to B$. Similar considerations apply for pointed parameterised spaces: either a functor $Y \colon B \to \catS_*$ or a space $Y \to B$ with a section $B \to Y$.

\subsection*{Acknowledgments} 
I am very grateful to Oscar Randal-Williams for first suggesting to use orthogonal calculus for maps into $\bP(V)$, and for many helpful discussions which kept this project afloat. I was supported by his ERC grant under the European Union's Horizon 2020 research and innovation programme (grant agreement No. 756444). Thanks to Samuel Muñoz-Echániz as well for excellent help with orthogonal calculus. Thanks to Greg Arone for making me aware of \cite{Malkiewich}, which simplified parts where I was reinventing the wheel. Finally, I was also supported by Dan Petersen's Wallenberg Scholar fellowship during the writing of this article.

\section{Unitary calculus}\label{section:unitary-calculus}

\subsection{Aide-mémoire on calculus}\label{subsection:aide-memoire}

Michael Weiss introduced his orthogonal calculus in \cite{Weiss} for functors on a certain topological category of vector spaces. In this paper we shall work with the unitary version of this calculus\footnote{Weiss' original work is directly applicable to this setting as done in \cite[Example~10.2]{Weiss}, but see also \cite{Taggart} for more details.} and only deal with functors
\[
    F \colon \catJ \to \catSp
\]
from the category $\catJ$ of finite-dimensional complex inner product spaces and linear isometries to the category of spectra $\catSp$. To such a functor this calculus associates for each $k \geq 1$ a $U(k)$-spectrum $\Theta_k F \in \catSp^{BU(k)}$, called the \emph{$k$th derivative}, and a tower of functors
\[
\begin{tikzcd}
& & \vdots \arrow[d] &                                              \\
& & T_2F(-) \arrow[d]   & \Theta_2 F(-) \otimes_{U(2)} \bS^{2\cdot -} \arrow[l] \\
& & T_1F(-) \arrow[d]   & \Theta_1 F(-) \otimes_{U(1)} \bS^{1\cdot -} \arrow[l]    \\
F(-) \arrow[rr] \arrow[rru] \arrow[rruu] \arrow[rruuu] & & T_0F(-) & \colim\limits_{n \to \infty} F(\bC^n) \arrow[l, equal]&          
\end{tikzcd}
\]
called the \emph{Weiss tower} of $F$. We explain a bit more: $\bS^{k \cdot V}$ is the suspension spectrum of the one-point compactification of $k \cdot V = \bC^k \otimes V$ on which the unitary group $U(k)$ acts naturally on the $\bC^k$ component, and the symbol $-\otimes_{U(k)}-$ denotes the homotopy orbits for the diagonal action of $U(k)$. For any $V \in \catJ$, the sequence
\[
    \Theta_k F \otimes_{U(k)} \bS^{kV} \lra T_kF(V) \lra T_{k-1}F(V)
\]
is a (co)fibre sequence of spectra, and the left term is called the $k$th layer of the tower. We recall a few basic definitions from the theory.
\begin{definition}[{\cite[Definition~5.1]{Weiss}}]\label{definition:polynomial}
Let $k \geq 0$. We say that a unitary functor $F \colon \catJ \to \catSp$ is \emph{$k$-polynomial} if the natural morphism
\[
    F(V) \lra \lim_{0 \neq W \subset \bC^{k+1}} F(V \oplus W) = \tau_kF(V)
\]
is an equivalence. In the formula, the limit is taken over the topological poset of non-zero subspaces of $\bC^{k+1}$ (see \cite[Proof of Lemma~e.3]{WeissErratum} for details on this limit).
\end{definition}
Let $\mathsf{Poly}^{\leq k}(\catSp^\catJ) \subset \catSp^\catJ$ be the full subcategory on those $k$-polynomial functors. This inclusion admits a left adjoint, and the unit applied to any unitary functor $F$ yields the $k$th polynomial approximation
\[
    F \lra T_kF
\]
appearing in the Weiss tower above. See \cite[Remark~5.2]{HahnYuan} for details. It will also be useful to know that $T_kF$ can be constructed as the colimit of 
\[
    F \overset{\rho}{\lra} \tau_kF \overset{\tau_k(\rho)}{\lra} \tau_k^2F \overset{\tau_k^2(\rho)}{\lra} \cdots
\]
where $\rho$ is the natural map of \cref{definition:polynomial}, see \cite[Theorem~6.3]{Weiss}. The layers $\mathrm{fib}(T_kF \to T_{k-1}F)$ have a special property which we recall:
\begin{definition}[{\cite[Definition~7.1 and Theorem~7.3]{Weiss}}]
A unitary functor $F \colon \catJ \to \catSp$ is \emph{homogeneous of degree $k$} if it is $k$-polynomial and $T_{k-1}F(V) \simeq 0$ for all $V \in \catJ$. In that case, there exists a spectrum $\Theta \in \catSp^{BU(k)}$ with an action of the unitary group $U(k)$ such that 
\[
    F(V) = (\Theta \otimes \bS^{kV})_{hU(k)} = \Theta \otimes_{U(k)} \bS^{kV}.
\]
Conversely any functor of this form is $k$-homogeneous.
\end{definition}

The following result, combined with the numerical range in the main theorem of \cite{Aumonier}, was the main impetus for this work:
\begin{proposition}[{\cite[Lemma~e.7]{Weiss} and \cite[Lemma~9.6]{Taggart}}]\label{proposition:numerical-observation}
Let $F \to G$ be a natural transformation of unitary functors. Suppose that there exist two constants $c \in \bZ$ and $d \in \bN$ such that $F(V) \to G(V)$ is $(2(d+1)\dim_\bC(V) + c)$-connected for any $V \in \catJ$. Then $T_dF \to T_dG$ is a levelwise equivalence. In particular $T_kF \simeq T_kG$ and $\Theta_kF \simeq \Theta_kG$ for all $k \leq d$.
\end{proposition}

\subsection{Shifts and colimits}

It will be sometimes convenient to shift unitary functors by precomposing with the endofunctor $\bC^N \oplus -$ of $\catJ$. We recall the effect on the Weiss tower.
\begin{lemma}\label{lemma:shifted-tower}
Let $F \colon \catJ \to \catSp$ be a unitary functor. Let $G = F(\bC^N \oplus -)$ be the shift of $F$. Then the Weiss tower of $G$ is obtained from that of $F$ via the formula:
\[
    T_kG(V) \simeq T_kF(\bC^N \oplus V) \quad \forall \ V \in \catJ.
\]
In particular, the derivatives are given by
\[
    \Theta_k G \simeq \bS^{k\bC^N} \otimes \Theta_k F
\]
as spectra with a $U(k)$-action.
\end{lemma}
\begin{proof}
Recall that the $k$th polynomial approximation of $F$ can be constructed as the sequential colimit
\[
    T_kF \simeq \colim(\tau_kF \lra \tau_k\tau_kF \lra \ldots)
\]
where
\[
    \tau_kF(V) = \lim_{0 \neq W \subset \bC^{k+1}} F(V \oplus W).
\]
We directly see that $\tau_kG(V) = \tau_kF(\bC^N \oplus V)$. By induction, the same holds for the $j$th composition $\tau_k^j$ and therefore the first claim follows after passing to the colimit. To identify the derivatives, we write the $k$th homogeneous layer of $G$:
\begin{align*}
    D_kG(V) &= \mathrm{fib}(T_kG(V) \lra T_{k-1}G(V)) \\
        &\simeq \mathrm{fib}(T_kF(\bC^N \oplus V) \lra T_{k-1}F(\bC^N \oplus V)) \\
        &\simeq \Theta_kF \otimes_{U(k)} \bS^{k(\bC^N \oplus V)} \\
        &\simeq (\Theta_kF \otimes \bS^{k\bC^N}) \otimes_{U(k)} \bS^{kV}. \qedhere
\end{align*}
\end{proof}

The following consequence of working with the stable category of spectra will be useful as well.
\begin{lemma}\label{lemma:polynomial-colimit-is-polynomial}
Let $F = \colim_{i \in I} F_i$ be a colimit of unitary functors to spectra which are $k$-polynomial. Then $F$ is also $k$-polynomial.
\end{lemma}
\begin{proof}
By definition $F$ is $k$-polynomial if the natural map
\[
    F(V) \lra \lim_{0 \neq W \subset \bC^{k+1}} F(U \oplus W)
\]
is an equivalence. The limit is finite in the $\infty$-categorical sense because the topological poset is compact, hence it commutes with colimits as $\catSp$ is a stable $\infty$-category.
\end{proof}

\subsection{Parameterised unitary calculus}\label{subsection:parameterised-unitary-calculus}

In our arguments later we will encounter unitary functors that are naturally expressed as the (suspension spectra of) total spaces of bundles of unitary functors. It will thus be convenient to study such functors fibrewise over their base $B$, i.e. first as functors $\catJ \to \catSp^B$ and then post-compose with the colimit functor $\catSp^B \to \catSp$. Unitary calculus for such parameterised unitary functors $\catJ \to \catSp^B$ works as in the unparameterised setting recalled above. We shall not dwell on the technicalities: we will only need a few basic observations and the definition of polynomiality which is identical to that recalled in \cref{definition:polynomial}. 

In this subsection, let $B$ be a space and $F \colon \catJ \to \catSp^B$ a unitary functor. For any point $i_b \colon \{b\} \hookrightarrow B$, we denote by $F_b = i_b^*F$ the restriction of $F$ to $b$. It is a unitary functor $F_b \colon \catJ \to \catSp$ because $i_b^* \colon \catSp^B \to \catSp$ is a functor.

\begin{lemma}\label{lemma:fibrewise-polynomiality}
The functor $F$ is $k$-polynomial if and only if $F_b$ is $k$-polynomial for every $b \in K$.
\end{lemma}
\begin{proof}
Recall that $F$ is $k$-polynomial if the natural morphism
\[
    F(V) \lra \lim_{0 \neq W \subset \bC^{k+1}} F(U \oplus W)
\]
is an equivalence. For every $b \in B$, we have an induced morphism
\[
    F_b(V) \lra i_b^* \left( \lim_{0 \neq W \subset \bC^{k+1}} F(U \oplus W) \right) \simeq \lim_{0 \neq W \subset \bC^{k+1}} F_b(U \oplus W)
\]
where we have used that $i_b^*$ commutes with limits. The ``only if" direction follows directly. The converse follows from the fact that the $i_b^*$ are jointly conservative (in fact, it suffices to verify the assumption for a single $b$ in each connected component of $B$).
\end{proof}

As a consequence, $i_b^*$ restricts to a functor between the categories of $k$-polynomial functors. In fact more is true:
\begin{lemma}
The natural morphism $T_k i_b^*F \to i_b^* T_k F$ is an equivalence.
\end{lemma}
\begin{proof}
This follows from the construction of the Taylor approximation as a sequential colimit
\[
    T_k F \simeq \colim(\tau_kF \lra \tau_k^2F \lra \tau_k^3 F \lra \ldots)
\]
and the fact that $i_b^*$ commutes with colimits and limits, so that $\tau_k i_b^*F \simeq i_b^* \tau_k F$.
\end{proof}

Let us now investigate the behaviour of the Weiss tower when applying the colimit functor $\catSp^B \to \catSp$. 
\begin{lemma}
If $F$ is $k$-polynomial then $\colim_B F$ is $k$-polynomial.
\end{lemma}
\begin{proof}
It follows from the definition of polynomiality and the fact that
\[
    \colim_{b \in B} \lim_{0 \neq W \subset \bC^{k+1}} F_b(U \oplus W) \simeq \lim_{0 \neq W \subset \bC^{k+1}} \colim_{b \in B} F_b(U \oplus W)
\]
because the limit is finite in the $\infty$-categorical sense and thus commutes with the colimit as $\catSp$ is stable.
\end{proof}

\begin{lemma}\label{lemma:parameterised-colimit-tower}
The natural morphism
\[
    T_k (\colim_{b \in B} F_b) \lra \colim_{b \in B} T_kF_b
\]
is an equivalence.
\end{lemma}
\begin{proof}
Again, this follows from the construction of the Taylor approximation and commuting a colimit with another colimit and a finite limit.
\end{proof}

\begin{lemma}\label{lemma:parameterised-colim-derivatives}
The derivative of the colimit is the colimit of the derivatives, i.e.
\[
    \Theta_k (\colim_{b \in B} F_b) \simeq \colim_{b \in B} \Theta_k F_b
\]
\end{lemma}
\begin{proof}
Recall that the $k$th layer functor is the fibre $\mathrm{fib}(T_kF \to T_{k-1}F)$. It is a finite limit hence commutes with the colimit over $B$. The result follows immediately.
\end{proof}

\section{Polynomiality of holomorphic maps}\label{section:polynomiality}

Let $X$ be a smooth connected complex projective variety. Any embedding $V \hookrightarrow W$ of complex vector spaces gives rise to an embedding $\bP(V) \hookrightarrow \bP(W)$ between the projectivisations. We thus obtain a unitary functor
\begin{align*}
    \catJ &\lra \catSp \\
    V &\longmapsto \Sigma^\infty_+ \Hol(X, \bP(V)),
\end{align*}
and the goal of this section is to study its Weiss tower. A holomorphic map $f \colon X \to \bP(V)$ of degree $\alpha = f^*(h) \in H^2(X;\bZ)$ is given by the data of a line bundle $\cL$ with $c_1(\cL) = \alpha$ together with a non-vanishing section of $\cL \otimes V$, up to equivalence. More precisely there is a map
\[
    \Hol_\alpha(X, \bP(V)) \lra \Pic^\alpha(X)
\]
from the space of holomorphic maps of degree $\alpha$ to the Picard variety of isomorphism classes of line bundles with first Chern class $\alpha$. The fibre above $[\cL] \in \Pic^\alpha(X)$ is the space $\Gammahol(\cL \otimes V - 0) / \bC^\times$ of never-vanishing holomorphic sections modulo the scalar action. Guided by this, we will first study the unitary functor
\[
    V \longmapsto \Sigma^\infty_+ \Gammahol(\cL \otimes V - 0)
\]
and then explain in \cref{subsection:polynomiality-full-mapping-space} how to deduce the result for the space of holomorphic maps. Functoriality is given by extending a section by zero.

\bigskip

Our main idea is to identify the space $\Gammahol(\cL \otimes V - 0)$ with a space of linear maps subject to conditions on their kernels. Related identifications, although different, appear in \cite{Gomezgonzales,,Crawford}. We shall then study an associated stratification by rank, and deduce our main results inductively using Miller's stable splitting \cite{Miller}. We will assume that $\cL$ is very ample, and let
\begin{equation}\label{equation:ample-embedding}
\begin{split}
    \iota \colon X &\longhookrightarrow \bP(H^0(X,\cL)^\vee) =  (H^0(X,\cL)^\vee \setminus 0) / \bC^\times \\
    x &\longmapsto \mathrm{ev}_x
\end{split}
\end{equation}
be the natural embedding and write $\iota X$ for its image. There are isomorphisms of vector spaces
\[
    \Gammahol(\cL \otimes V) = H^0(X, \cL \otimes V) \cong H^0(X, \cL) \otimes V \cong \Hom(H^0(X, \cL)^\vee, V)
\]
that are natural in $V$. We can thus see any holomorphic section of $\cL \otimes V$ as a linear map $H^0(X, \cL)^\vee \to V$. The fact that the section vanishes nowhere on $X$ translates to a condition on the kernel of that linear map. More precisely:
\begin{equation}\label{equation:correspondence-sections-linear-maps}
    \Gammahol(\cL \otimes V - 0) \cong \set{A \in \Hom(H^0(X, \cL)^\vee, V)}{\bP(\ker A) \cap \iota X = \emptyset}.
\end{equation}
This is the point of view that we will adopt in the remainder of this section.

\subsection{The rank filtration on linear maps}
We recall some basic linear algebra and its consequences on manifolds of linear maps of bounded-below rank. Let $V, W$ be two inner product complex vector spaces of dimensions $n$ and $m$ respectively. By lower semicontinuity of the rank function, for any $0 \leq r \leq \min(n,m)$ the subspace
\[
    \Hom^{\geq r}(V,W) := \set{A \in \Hom(V,W)}{\rank A \geq r} \subset \Hom(V,W)
\]
is open. Furthermore, the subspace
\[
    \Hom^r(V,W) := \set{A \in \Hom(V,W)}{\rank A = r} \subset \Hom(V,W)
\]
is a smooth submanifold of complex codimension $(n-r)(m-r)$, see e.g. \cite[Exercise~13 p.~27]{GuilleminPollack}. Let us also recall what its normal bundle is:
\begin{lemma}\label{lemma:normal-bundle-rank-matrices}
Let $\ker \colon \Hom^r(V,W) \to \Grass(n-r,V)$ and $\image \colon \Hom^r(V,W) \to \Grass(r,W)$ be the kernel and image maps. Let $\cH$ be the hyperplane bundle on $\Grass(n-r,V)$ and $\cQ$ be the quotient bundle on $\Grass(r,W)$. Then the normal bundle of $\Hom^r(V,W) \subset \Hom(V,W)$ is given by the vector bundle $\underline{\Hom}(\ker^* \cH, \image^* \cQ)$.
\end{lemma}
\begin{proof}
The space of linear maps $\Hom(V,W)$ is affine, hence its tangent bundle is the trivial bundle $\underline{\Hom}(\underline{V},\underline{W})$. On its restriction to $\Hom^r(V,W)$ define
\begin{equation}\label{equation:normalmap}
\begin{split}
    \underline{\Hom}(\underline{V},\underline{W})|_{\Hom^r(V,W)} &\lra \underline{\Hom}(\ker^* \cH, \image^* \cQ) \\
    (A \in \Hom^r(V,W), \varphi \in \Hom(V,W)) &\longmapsto (A, \psi \colon \ker A \subset V \overset{\varphi}{\to} W \twoheadrightarrow W / \image A).
\end{split}
\end{equation}
This is visibly a surjective morphism of vector bundles on $\Hom^r(V;W)$. Let us show that its kernel contains the tangent bundle $T\Hom^r(V,W)$. By the definition using tangent curves, we have an identification
\[
    T\Hom^r(V,W) \cong \set{(A,\varphi)}{\rank (A+t\varphi) = r \text{ for } t>0 \text{ small enough}} \subset \underline{\Hom}(\underline{V},\underline{W})|_{\Hom^r(V,W)}.
\]
Let $(A,\varphi) \in T\Hom^r(V,W)$. If $v \in \ker A$, then $\varphi(v) \in \image A$. Indeed, if it were not the case then $\rank(A+t\varphi) \geq r+1$ because $\image(A+t\varphi) \supset \image A \oplus \mathrm{span}(\varphi(v))$ for small $t > 0$. This shows our claim. Now, by counting the dimensions, the kernel of~\eqref{equation:normalmap} is in fact equal to $T\Hom^r(V,W)$. The cokernel is thus the normal bundle.
\end{proof}

\subsection{Polynomiality of holomorphic sections}

To lighten the notation, let us choose an isomorphism $\bC^N \cong H^0(X, \cL)^\vee$, and write
\[
    \Phi(V) = \set{A \in \Hom(\bC^N, V)}{\bP(\ker A) \cap \iota X = \emptyset} \cong \Gammahol(\cL \otimes V - 0).
\]
The goal of this section is to show:
\begin{theorem}\label{theorem:holomorphic-polynomiality}
The unitary functor $\Sigma^\infty_+ \Phi$ is polynomial of degree $N = \dim H^0(X,\cL)$.
\end{theorem}

Our proof uses various induction steps. The reader who wants to see the proof ``in action" in an example where these steps are kept to a minimum, but all the key ideas remain, is invited to look at \cref{subsection:p1p1-case}.

\begin{proof}
We will prove the theorem after a few reductions. We explain them here and reference future lemmas to finish the proof. Applying the rank filtration to $\Phi(V)$ produces a stratification
\[
    \Phi(V) = \Phi_0(V) \cup \Phi_1(V) \cup \ldots \cup \Phi_N(V)
\]
into disjoint submanifolds, where
\[
    \Phi_r(V) = \set{A \in \Phi(V)}{\dim \ker A = r}.
\]
This decomposition is functorial in $V$ because post-composing a linear map $\bC^n \to V$ with an embedding $V \hookrightarrow W$ does not change the dimension of the kernel.
Let us write $\nu_r(V)$ for the normal bundle of $\Phi_r(V)$ inside $\Phi_{\leq r}(V) = \bigcup_{i=0}^r \Phi_i(V)$. For each $r$ we have a homotopy pushout square of unitary functors
\[
\begin{tikzcd}
S(\nu_r(V)) \arrow[d] \arrow[r] & D(\nu_r(V)) \simeq \Phi_r(V) \arrow[d] \\
\Phi_{<r}(V) \arrow[r]          & \Phi_{\leq r}(V)     
\end{tikzcd}
\]
where $S(-)$ and $D(-)$ denote the sphere and disc bundle respectively. To show our \cref{theorem:holomorphic-polynomiality}, it suffices by \cref{lemma:polynomial-colimit-is-polynomial} to show that each of $\Sigma^\infty_+S(\nu_r(V))$, $\Sigma^\infty_+\Phi_r(V)$ and $\Sigma^\infty_+\Phi_{<r}(V)$ is $N$-polynomial. By induction, it suffices to show that for each $r=0,\ldots,N$ the functors $\Sigma^\infty_+S(\nu_r(V))$ and $\Sigma^\infty_+\Phi_r(V)$ are $N$-polynomial.
We will make use of a further reduction. Let us write
\[
    \Grass(r, N; \iota X) =  \set{P \in \Grass(r, \bC^N)}{P \cap \iota X = \emptyset} \subset \Grass(r, \bC^N)
\]
for the locus of $r$-planes not intersecting $\iota X$.

\begin{lemma}\label{lemma:fibrewise-functor}
Let $\cQ_r$ be the restriction of tautological quotient bundle on the Grassmannian to the locus $\Grass(r, N; \iota X)$. There is a homeomorphism of unitary functors to spaces over $\Grass(r, N; \iota X)$
\[
    \Phi_r(V) \cong \underline{\Mono}(\cQ_r, \underline{V})
\]
where $\underline{V}$ denotes the trivial vector bundle with fibre $V$, and $\underline{\Mono} \subset \underline{\Hom}$ denotes the bundle of linear injections.
\end{lemma}
\begin{proof}
The kernel map exhibits $\Phi_r(V)$ as a space over $\Grass(r,N; \iota X)$. The natural map over $\Grass(r,N; \iota X)$
\begin{align*}
    \Phi_r(V) &\lra \underline{\Mono}(\cQ_r, \underline{V}) \\
    A &\longmapsto (\ker A, \tilde{A} \colon \bC^N / \ker A \hookrightarrow V)
\end{align*}
is a homeomorphism as the induced map on the fibre above a point $K$ is the natural isomorphism
\[
    \set{A \in \Phi_r(V)}{\ker A = K} \overset{\cong}{\lra} \Mono(\bC^N/K, V). \qedhere
\]
\end{proof}

Given $K \in \Grass(r,N; \iota X)$, we denote by $\Phi_r^K(V) \subset \Phi_r(V)$ the subspace of those linear maps $A$ with $\ker A = K$, and by $S(\nu_r(V))_K \subset S(\nu_r(V))$ the restriction of the sphere bundle to $\Phi_r^K(V)$. By \Cref{lemma:fibrewise-polynomiality}, to show that $\Sigma^\infty_+\Phi_r(V)$ and $\Sigma^\infty_+S(\nu_r(V))$ are $N$-polynomial it suffices to show that $\Sigma^\infty_+\Phi_r^K(V)$ and $\Sigma^\infty_+S(\nu_r(V))_K$ are $N$-polynomial for every $K \in \Grass(r,N; \iota X)$. The first result is proved in \cref{lemma:main-disc-polynomial} below, while the second follows from \cref{corollary:induction-pushout-sphere-polynomial} and the equivalence
\begin{align*}
    S(\nu_r(V))_K &\simeq \set{(A, \varphi) \in \Phi_r^K(V) \times \Hom(K, V / \image A)}{\varphi \neq 0} \\
    &= \set{(A, \varphi) \in \Mono(\bC^N/K, V) \times \Hom(K, V / \image A)}{\varphi \neq 0}.
\end{align*}
which is a consequence of \cref{lemma:normal-bundle-rank-matrices}.
\end{proof}

\begin{lemma}\label{lemma:main-disc-polynomial}
The functor $\Sigma^\infty_+\Phi_r^K$ is $(N-r)$-polynomial, hence $N$-polynomial.
\end{lemma}
\begin{proof}
By Gram--Schmidt there is a functorial homotopy equivalence 
\[
    \Phi_r^K(V) \simeq \Hom_\catJ(\bC^N/K, V)
\]
with the Stiefel manifold of $(N-r)$-frames in $V$. Miller's stable splitting \cite{Miller} shows that $\Sigma^\infty_+ \Hom_\catJ(\bC^N/K, V)$ is $(N-r)$-polynomial, see \cite{Arone} for the point of view of unitary calculus on this result.
\end{proof}

We now turn our attention to the sphere bundle. For notational reasons, it will be convenient for our recursive arguments to consider a general situation. Let $K \subset \bC^N$ be a linear subspace of dimension $k$ and define
\[
    M_K(V) = \left\{ (A, \varphi) \in \Mono(\bC^N/K, V) \times \Hom(K, V/\image A) \right\}.
\]
There is a stratification of $M_K$ by submanifolds
\[
    M_K(V) = M_K^0(V) \cup M_K^1(V) \cup \ldots \cup M_K^k(V)
\]
where
\[
    M_K^i(V) = \set{(A, \varphi) \in M_K(V)}{\dim \ker \varphi = i}.
\]
Let $M_K^{\leq i}(V) = \bigcup_{j=0}^i M_K^j(V)$ and let $\mu_K^i(V)$ be the normal bundle of the inclusion $M_K^i(V) \subset M_K^{\leq i}(V)$.

\begin{lemma}\label{lemma:induction-disc-polynomial}
The functor $\Sigma^\infty_+M_K^i(V)$ is $(N-i)$-polynomial.
\end{lemma}
\begin{proof}
The space $M_K^i(V)$ fibres over the Grassmannian $\Grass(i, K)$ via the map taking the kernel of $\varphi$. By \cref{lemma:fibrewise-polynomiality} it suffices to show that, for any $P \in \Grass(i,K)$, the functor
\[
    V \longmapsto \Sigma^\infty_+ \set{(A, \varphi) \in M_K^i(V)}{\ker \varphi = P}
\]
is $(N-i)$-polynomial. We have equivalences
\begin{align*}
    \set{(A, \varphi) \in M_K^i(V)}{\ker \varphi = P} &\cong \set{(A, \tilde{\varphi})}{A \in \Mono(\bC^N/K,V), \ \tilde{\varphi} \in \Mono(K/P, V/\image A)} \\
    &\simeq \Hom_\catJ(\bC^N/K \oplus K/P, V).
\end{align*}
And by Miller's stable splitting, $V \mapsto \Sigma^\infty_+\Hom_\catJ(\bC^N/K \oplus K/P, V)$ is $(N-i)$-polynomial.
\end{proof}

\begin{lemma}\label{lemma:induction-sphere-polynomial}
The functor $\Sigma^\infty_+S(\mu_K^i(V))$ is $N$-polynomial.
\end{lemma}
\begin{proof}
We will show the result inductively on the dimension $k$ of $K$. If $k=0$, the result is clear because $\mu_K^i$ has rank $0$, hence the sphere bundle is empty. Suppose now that $K$ has dimension $k > 0$ and the result is true for lower dimensions. The space $S(\mu_K^i(V))$ fibres over the Grassmannian $\Grass(i, K)$ via the map taking the kernel of $\varphi$. By \cref{lemma:normal-bundle-rank-matrices}, the fibre above $P \in \Grass(i,K)$ is homotopy equivalent to the space
\[
    S(\mu_K^i(V))_P = \set{(A, \varphi, \phi) \in M_K^i(V) \times \Hom(P, V/\image (A+\varphi))}{\phi \neq 0}.
\]
By \cref{lemma:fibrewise-polynomiality} it is enough to show that $\Sigma^\infty_+S(\mu_K^i(V))_P$ is $N$-polynomial for every $P \in \Grass(i,K)$.
We may further stratify this space by submanifolds according to the dimension of the kernel of $\phi$:
\[
    S(\mu_K^i(V))_P = S(\mu_K^i(V))_P^0 \cup S(\mu_K^i(V))_P^1 \cup \ldots \cup S(\mu_K^i(V))_P^{i-1}
\]
where
\[
    S(\mu_K^i(V))_P^j = \set{(A, \varphi, \phi) \in S(\mu_K^i(V))_P}{\dim \ker \phi = j}.
\]
Let $S(\mu_K^i(V))_P^{\leq k} = \bigcup_{j=0}^k S(\mu_K^i(V))_P^j$ and $\eta_{K,P}^{i,j}(V)$ be the normal bundle of the inclusion $S(\mu_K^i(V))_P^j \subset S(\mu_K^i(V))_P^{\leq j}$. For any $0 \leq j \leq i-1$, we have a homotopy pushout square
\[
\begin{tikzcd}
{S(\eta_{K,P}^{i,j}(V))} \arrow[d] \arrow[r] & {D(\eta_{K,P}^{i,j}(V)) \simeq S(\mu_K^i(V))_P^j} \arrow[d] \\
S(\mu_K^i(V))_P^{< j} \arrow[r]       & S(\mu_K^i(V))_P^{\leq j} 
\end{tikzcd}
\]
Hence, by induction, to show that $\Sigma^\infty_+S(\mu_K^i(V))_P$ is $N$-polynomial, it is enough to show that $\Sigma^\infty_+S(\mu_K^i(V))_P^j$ and $\Sigma^\infty_+S(\eta_{K,P}^{i,j}(V))$ are $N$-polynomial. Let us begin with the former: we have an equivalence
\[
    S(\mu_K^i(V))_P^j \simeq M_{P}^j(V), \quad (A, \varphi, \phi) \mapsto (A+\varphi, \phi).
\]
Hence $\Sigma^\infty_+S(\mu_K^i(V))_P^j$ is $(N-j)$-polynomial by \cref{lemma:induction-disc-polynomial}. We now turn our attention to the sphere bundle $S(\eta_{K,P}^{i,j}(V))$. This space fibres over the Grassmannian $\Grass(j, P)$ via the map taking the kernel of $\phi$. By \cref{lemma:normal-bundle-rank-matrices}, the fibre above $L \in \Grass(j,P)$ is homotopy equivalent to the space
\[
    S(\eta_{K,P}^{i,j}(V))_L \simeq \set{(A, \varphi, \phi, \psi) \in S(\mu_K^i(V))_P^j \times \Hom(L, V/(\image A + \varphi + \phi))}{\psi \neq 0}.
\]
By \cref{lemma:fibrewise-polynomiality} again, it is enough for us to show that $\Sigma^\infty_+S(\eta_{K,P}^{i,j}(V))_L$ is $N$-polynomial to conclude that $\Sigma^\infty_+S(\eta_{K,P}^{i,j}(V))$ is $N$-polynomial. But this follows by our inductive assumption as $\dim L < \dim P$.
\end{proof}

\begin{corollary}\label{corollary:induction-pushout-sphere-polynomial}
For any $0 \leq i \leq k$, the functor $\Sigma^\infty_+M_K^{\leq i}(V)$ is $N$-polynomial.
\end{corollary}
\begin{proof}
This follows from the homotopy pushout square
\[
\begin{tikzcd}
S(\mu_K^i(V)) \arrow[d] \arrow[r] & D(\mu_K^i(V)) \simeq M_K^{i}(V) \arrow[d] \\
M_K^{< i}(V) \arrow[r]            & M_K^{\leq i}(V)                          
\end{tikzcd}
\]
and \cref{lemma:induction-disc-polynomial,lemma:induction-sphere-polynomial}.
\end{proof}

\subsection{Polynomiality of holomorphic maps}\label{subsection:polynomiality-full-mapping-space}

There are two distinct steps to go from $\Gammahol(\cL \otimes V -0)$ to $\Hol_\alpha(X, \bP(V))$. The first is to understand how the polynomiality of $\Sigma^\infty_+\Gammahol(\cL \otimes V -0)$ implies that of $\Sigma^\infty_+(\Gammahol(\cL \otimes V -0)/\bC^\times)$: this will follow from general manipulations in unitary calculus and we explain this first. The second requires to understand how varying $\cL$ in $\Pic^\alpha(X)$ affects the unitary functor of sections.

\subsubsection{The canonical \texorpdfstring{$U(1)$}{U(1)}-action}\label{subsubsection:canonical-u1-action}
Any unitary functor $F$ is canonically equipped with an action by the unitary group $U(1)$: an element $\lambda \in U(1)$ acts on $F$ via the natural transformation
\[
    F(V) \overset{F(\lambda \cdot)}{\lra} F(V)
\]
where $\lambda \cdot$ is the linear morphism given by multiplication by $\lambda$. We thus obtain a functor
\[
    (-)_{hU(1)} \colon \catFun(\catJ, \catSp) \lra \catFun(\catJ, \catSp)
\]
sending a unitary functor $F$ to its homotopy orbits $F_{hU(1)}(V) = F(V)_{hU(1)}$. We record its effect on the tower:
\begin{lemma}\label{lemma:natural-u1-action}
There is a natural equivalence $(T_kF)_{hU(1)} \simeq T_k(F_{hU(1)})$. In particular, if $F$ if $k$-polynomial then so is $F_{hU(1)}$. On derivatives there is an induced equivalence $\Theta_k (F_{hU(1)}) \simeq (\Theta_kF)_{hU(1)}$ where the second action is through the diagonal inclusion $U(1) \subset U(k)$.
\end{lemma}
\begin{proof}
The first statement follows from the formula for $T_k(-)$ and commuting the finite limit appearing in the formula for $\tau_k$ with the homotopy orbits. For the second statement, recall that the $k$th layer is given by
\[
    D_kF(V) \simeq \Theta_kF \otimes_{U(k)} \bS^{kV}
\]
so that
\[
    D_k(F_{hU(1)})(V) \simeq (D_kF(V))_{hU(1)} \simeq (\Theta_kF \otimes_{U(k)} \bS^{kV})_{hU(1)}
\]
where the first equivalence is a consequence of the first part of the lemma, and $U(1)$ acts on $V$. But $U(k)$ acts diagonally and the $U(1)$ action commutes with it, hence 
\[
    (\Theta_kF \otimes_{U(k)} \bS^{kV})_{hU(1)} \simeq (\Theta_kF)_{hU(1)} \otimes_{U(k)} \bS^{kV}. \qedhere
\]
\end{proof}

The scalar action of $U(1) \subset \bC^\times$ on $\Gammahol(\cL \otimes V -0)$ coincides with the canonical action on the unitary functor. We thus obtain:
\begin{corollary}
Let $N = \dim H^0(X,\cL)$. The unitary functor
\[
    V \longmapsto \Sigma^\infty_+(\Gammahol(\cL \otimes V -0)/\bC^\times)
\]
is $N$-polynomial.
\end{corollary}
\begin{proof}
The $\bC^\times$-action is free so the orbits are homotopy orbits. Combined with the homotopy equivalence $U(1) \simeq \bC^\times$, we get
\[
    \Gammahol(\cL \otimes V -0)/\bC^\times \simeq \Gammahol(\cL \otimes V -0)_{hU(1)}
\]
and the result follows from \cref{theorem:holomorphic-polynomiality}.
\end{proof}

\subsubsection{The Picard parameter}\label{subsubsection:picard-parameter}

Let $\alpha \in H^2(X;\bZ)$ and let $n = \dim_\bC X$. The space of holomorphic maps of degree $\alpha$ maps down to the Picard variety $\Pic^\alpha(X)$ of isomorphism classes of holomorphic line bundles of first Chern class $\alpha$:
\[
    \Hol_\alpha(X,\bP(\bC^{n+1} \oplus V)) \overset{\pi}{\lra} \Pic^\alpha(X)
\]
with fibre above $[\cL] \in \Pic^\alpha(X)$ the section space $\Gammahol(\cL \otimes (\bC^{n+1} \oplus V) - 0) / \bC^\times$. Notice that we have shifted our unitary functor by $\bC^{n+1} \oplus -$ to avoid empty spaces for convenience, but this has little effect on the Weiss tower as explained in \cref{lemma:shifted-tower}. We would like to view the space of holomorphic maps as a parameterised unitary functor over $\Pic^\alpha(X)$ and apply the results just established in combination with those of \cref{subsection:parameterised-unitary-calculus}. Unfortunately the projection map $\pi$ is not a fibration, and one would thus need to first replace it by a fibration to obtain the right parameterised functor. However the outcome of such a procedure is a functor that is only related to the space of holomorphic maps in a range, which is not enough to understand the full Weiss tower. Instead, we shall explain how to adapt the proof of \cref{theorem:holomorphic-polynomiality} to incorporate variations of $\cL$ in $\Pic^\alpha(X)$. As this extra complication introduces even more variables, we will be content with only indicating the most salient modifications.
\begin{theorem}\label{theorem:final-holomorphic-polynomiality}
Let $X$ be a connected smooth projective complex variety of dimension $n$. Let $\alpha \in H^2(X;\bZ)$ be such that $\alpha - c_1(K_X)$ is ample. Let $N = \dim H^0(X,\cL)$ where $\cL$ is any\footnote{By the Kodaira vanishing theorem this constant $N$ does not depend on the chosen $\cL$ with $c_1(\cL) = \alpha$.} very ample holomorphic line bundle with $c_1(\cL) = \alpha$. Then the unitary functor
\[
    V \longmapsto \Sigma^\infty_+ \Hol_\alpha(X,\bP(\bC^{n+1} \oplus V))
\]
is $N$-polynomial.
\end{theorem}
\begin{proof}
Let $\cP$ be a Picard line bundle on $\Pic^\alpha(X) \times X$, see \cite[Exercise~4.3]{Kleiman}. Let $p \colon \Pic^\alpha(X) \times X \to \Pic^\alpha(X)$ be the first projection. By cohomology and base change, together with the Kodaira vanishing theorem and our assumption on the ampleness of $\alpha - c_1(K_X)$, the pushed forward sheaf $p_*\cP$ is a vector bundle. There is a natural embedding
\[
    \Pic^\alpha(X) \times X \overset{\epsilon}{\longhookrightarrow} \bP\big( (p_*\cP)^\vee \big)
\]
generalising the embedding~\eqref{equation:ample-embedding} over $\Pic^\alpha(X)$. As in \cref{equation:correspondence-sections-linear-maps}, we observe that
\[
    \Hol_\alpha(X,\bP(\bC^{n+1} \oplus V)) \cong \set{ ([\cL], A) \in \underline{\Hom}((p_*\cP)^\vee, \underline{\bC^{n+1} \oplus V})}{\bP(\ker A) \cap \epsilon(\{[\cL]\} \times X) = \emptyset} / \bC^\times.
\]
Let us write
\[
    \Psi(V) = \set{ ([\cL], A) \in \underline{\Hom}((p_*\cP)^\vee, \underline{\bC^{n+1} \oplus V})}{\bP(\ker A) \cap \epsilon(\{[\cL]\} \times X) = \emptyset}.
\]
By \cref{lemma:natural-u1-action}, our theorem follows if we simply show that
\[
    V \longmapsto \Sigma^\infty_+ \Psi(V)
\]
is $N$-polynomial. As in the proof of \cref{theorem:holomorphic-polynomiality}, the rank filtration produces a stratification by smooth submanifolds
\[
    \Psi(V) = \Psi_0(V) \cup \Psi_1(V) \cup \ldots \cup \Psi_N(V)
\]
with $\Psi_r(V) = \set{([\cL], A) \in \Psi(V)}{\dim \ker A = r}$. Let $\nu_r(V)$ be the normal bundle of $\Psi_r(V)$ inside $\Psi_{\leq r}(V) = \bigcup_{i=0}^r \Psi_i(V)$. There is a homotopy pushout square
\[
\begin{tikzcd}
S(\nu_r(V)) \arrow[d] \arrow[r] & D(\nu_r(V)) \simeq \Psi_r(V) \arrow[d] \\
\Psi_{<r}(V) \arrow[r]          & \Psi_{\leq r}(V)     
\end{tikzcd}
\]
We can show again by induction that $\Sigma^\infty_+\Psi_{\leq r}(V)$ is $N$-polynomial. More precisely, for $r=0,\dotsc,N$ let
\[
    \Grass(r, (p_*\cP)^\vee; \epsilon) = \set{([\cL], P) \in \Grass(r,(p_*\cP)^\vee)}{P \cap \epsilon(\{[\cL]\} \times X) = \emptyset)},
\]
an open subset of the Grassmannian bundle $\Grass(r,(p_*\cP)^\vee)$ on $\Pic^\alpha(X)$. The kernel map gives a fibre bundle
\[
    \Psi_r(V) \overset{\ker}{\lra} \Grass(r, (p_*\cP)^\vee; \epsilon)
\]
with fibre above $([\cL], P)$ the space of injective linear maps $\Mono(H^0(X,\cL)^\vee / P, \bC^{n+1} \oplus V)$. By applying \cref{lemma:main-disc-polynomial} to the fibres, \cref{lemma:fibrewise-polynomiality} shows that $\Sigma^\infty_+\Psi_r(V)$ is $N$-polynomial. Likewise for the sphere bundle: we apply \cref{lemma:induction-sphere-polynomial} to the fibres as done at the end of the proof of \cref{theorem:holomorphic-polynomiality}.
\end{proof}

\section{Goodwillie calculus for continuous sections}\label{section:continuous}

In this section, we only assume that $X$ is a finite CW complex. Our goal now is to study the continuous analogue of the unitary functor of the previous section. More precisely, let $\cL$ be a complex line bundle on $X$. We shall study the unitary functor of continuous sections
\[
    V \longmapsto \Sigma^\infty_+ \Gammacon(\cL \otimes V - 0).
\]
Our strategy is to write it as the composition
\[
    \catJ \lra \catS^X \lra \catSp
\]
of $V \mapsto \cL \otimes V - 0$ and $E \mapsto \Sigma^\infty_+\Gammacon(E)$. This allows us to first study the second functor using Goodwillie calculus, which has been more extensively developed in the literature, and then deduce consequences for the composite unitary functor.

We begin this section by some recollections on parameterised spectra, mostly to fix our notation. We then recall Malkiewich's work on section spaces and its consequences for the Goodwillie tower. Finally we pass from the Goodwillie tower to the Weiss tower.

\subsection{Reminders on parameterised spectra}

We recall a few basic facts about parameterised spectra using the $\infty$-categorical framework, but as we shall only use basic facts any model would do. We recommend \cite[Appendix]{Land} for a nice introduction. The category of parameterised spectra over a space $B$ is the functor category $\catSp^B = \catFun(B, \catSp)$. It inherits a symmetric monoidal structure $-\otimes-$ from the category of spectra. The adjunction $\Sigma^\infty \colon \catS_* \rightleftharpoons \catSp \colon \Omega^\infty$ yields an adjunction 
\[
\begin{tikzcd}
{\Sigma^\infty_B \colon \catS_*^B = \catFun(B, \catS_*)} \arrow[r, "\Sigma^\infty_B", shift left] & {\catFun(B,\catSp) = \catSp^B : \Omega^\infty_B} \arrow[l, "\Omega^\infty_B", shift left]
\end{tikzcd}
\]
by post-composition of functors. Given a map of spaces $f \colon A \to B$, there are three functors $f^*, f_!, f_*$ given respectively by restriction and left and right Kan extension. They form a pair of adjunctions $f_! \dashv f^*$ and $f^* \dashv f_*$:
\[
\begin{tikzcd}
\catSp^A \arrow[r, "f_!", shift left=3] \arrow[r, "f_*"', shift right=3] & \catSp^B \arrow[l, "f^*" description]
\end{tikzcd}
\]
For example, if $G$ is an $\bE_1$-group in spaces and $r \colon BG \to *$ is the unique map to a point, then any spectrum $E \in \catSp^{BG}$ can be seen a spectrum with a $G$ action and $r_!(E)$ is the homotopy orbits spectrum. We will often write $r_!(E) \simeq E_{hG}$. More generally, if $f \colon G \to H$ is a map of $\bE_1$-groups, we will write $\Ind_G^H$ for the functor $(Bf)_! \colon \catSp^{BG} \to \catSp^{BH}$.
We will later need to use the external tensor product of parameterised spectra. If $E \in \catSp^A$ and $F \in \catSp^B$ are parameterised spectra, their external tensor product is the parameterised spectrum on $A \times B$ given by
\[
    E \boxtimes F \colon A \times B \overset{E \times F}{\lra} \catSp \times \catSp \overset{\otimes}{\lra} \catSp.
\]
This generalises directly to finitely many tensor factors. In particular the $k$-fold external tensor power $E^{\boxtimes k}$ is a parameterised spectrum over the product space $B^k$. In that particular case where all the factors are equal, there is an obvious $\Sigma_k$ action by permuting the factors and this construction can be refined to a parameterised spectrum over $B^k$ with a $\Sigma_k$ action. In other words $E^{\boxtimes k}$ can be lifted to a parameterised spectrum over the Borel construction $B^k \sslash \Sigma_k$. We shall use the same symbol for the lift.

\subsection{Recollections on Malkiewich's tower for section spaces}

We recall here and specialise to our situation the results of \cite[Section~3]{Malkiewich}. Let $E \in \catS_*^X$ and consider the functor
\begin{align*}
    G_E \colon (\catS^X)^\opp &\lra \catSp \\
    (B \overset{p}{\to}X) &\longmapsto \Sigma^\infty_+ \Gammacon(p^*E \to B)
\end{align*}
where we write $p$ for the structure map of the space $B$ over $X$. In \cite[Theorem~1]{Malkiewich}, Malkiewich constructs a tower of functors 
\[
    G_E \lra \cdots \lra A_kG_E \lra A_{k-1}G_E \lra \cdots \lra A_1G_E \lra A_0G_E
\]
functorially in $E$. (More generally, the tower is functorial in functors from $(\catS^X)^\opp$ to $\catSp$.) The term $A_kG_E$ stands for the $k$th \emph{a}pproximation of $G_E$ and is a $k$-excisive functor in the sense of \cite[Definition~2.2]{Malkiewich}. (We shall not need the general definition and will only state its consequences later as needed.)

The functors $A_kG_E$ have very explicit formulas which we will shortly reproduce from \cite[Section~3]{Malkiewich} after introducing some notation. Let $\underline{i} = \{1,2,\ldots,i\}$ and write $\catEpi_k$ for the category with objects $\underline{i}$ for $0 \leq i \leq k$ (by convention $\underline{0} = \emptyset$) and morphisms the surjective maps. Given $B \in \catS^X$ a space over $X$, its external smash product $B \overline{\wedge} B$ is the space over $X \times X$ whose fibre over $(x_1,x_2)$ is the smash product $B_{x_1} \wedge B_{x_2}$. More generally, the iterated external smash product $B^{\overline{\wedge}i}$ is a space over $X^i$. In the formula that follows, we implicitly replace $B$ and $E$ by homotopy equivalent spaces satisfying the (co)fibrancy conditions explained in \cite[Definitions~3.1, 3.2, 3.3]{Malkiewich}. Define
\[
    \Map_{(\catEpi_k^\opp, \{X^\bullet\})}(B^{\overline{\wedge} \bullet}, \Sigma^\infty_{X^\bullet}E^{\overline{\wedge} \bullet})
\]
to be the sequential spectrum whose $m$th level is the space of collections of maps\footnote{See \cite[Remark~3.5]{Malkiewich} concerning $f_0$. We deviate slightly from \cite[Section~3]{Malkiewich} as we consider the suspension with added basepoint $\Sigma^\infty_+$ instead of $\Sigma^\infty$.}
\[
    (f_0, \dotsc, f_n) \quad \text{with} \ f_i \colon B^{\overline{\wedge} i} \to \Sigma^m_{X^i}E^{\overline{\wedge} i} \text{ over } X^i
\]
such that each surjective map $\underline{j} \twoheadrightarrow \underline{i}$ in $\catEpi_k$ gives a commuting square
\[
\begin{tikzcd}
\Sigma^k_{X^i}E^{\overline{\wedge} i} \arrow[r]    & \Sigma^k_{X^j}E^{\overline{\wedge} j}     \\
B^{\overline{\wedge} i} \arrow[r] \arrow[u, "f_i"] & B^{\overline{\wedge} j} \arrow[u, "f_j"']
\end{tikzcd}
\]
The upshot of \cite[Section~3]{Malkiewich} is summarised in our case by:
\begin{theorem}[{\cite[Theorem~3.16]{Malkiewich}}]\label{theorem:malkiewich-tower}
The $k$th approximation of $G_E$ is given by
\[
    A_kG_E(B) \simeq \Map_{(\catEpi_k^\opp, \{X^\bullet\})}(B^{\overline{\wedge} \bullet}, \Sigma^\infty_{X^\bullet}E^{\overline{\wedge} \bullet}).
\]
\end{theorem}


As a consequence, we obtain a formula for the layers of Malkiewich's tower evaluated at $X$:
\begin{corollary}[Malkiewich]\label{corollary:malkiewich-layers}
Let $E \in \catS_*^X$ and write $(\Sigma^\infty_X E)^{\boxtimes k} \in \catSp^{X^k \sslash \Sigma_k}$ for its external tensor power with its natural permutation action. Let $\omega$ denote the composition $\Conf_k(X) \hookrightarrow X^k \to X^k \sslash \Sigma_k$. Then there is an equivalence of spectra
\[
    \mathrm{fib}(A_kG_E \to A_{k-1}G_E)(X) \simeq \big( \Sigma^\infty_X E^{\boxtimes k} \otimes \bS^{-TX^k} \otimes \omega_!\bS^0 \big)_{h(\Omega X \wr \Sigma_k)}
\]
where $\Omega X \wr \Sigma_k$ acts obviously on each factor hence diagonally on the tensor product.
\end{corollary}
\begin{proof}
We need to explain how to translate from Malkiewich's notation into our own: it is simply a manipulation of functors in parameterised spectra. Let $p \colon \Conf_k(X) \to \UConf_k(X) = \Conf_k(X) / \Sigma_k$ be the projection, and $r \colon \UConf_k(X) \to *$ and $q \colon X^k \sslash \Sigma_k \to *$ be the unique maps to a point. The formula given at the top of page~544 in \cite{Malkiewich} is the Thom spectrum
\[
    \big( ((E^{\overline{\wedge} k})|_{\Conf_k(X)})_{\Sigma_k} \big)^{-T\UConf_k(X)}
\]
where $(-)|_{\Conf_k(X)}$ denotes the restriction along the inclusion $\Conf_k(X) \hookrightarrow X^k$. The external smash power $E^{\overline{\wedge} k}$ is an object of $\catS_*^{X^k \sslash \Sigma_k}$ (i.e. a parameterised pointed space with a $\Sigma_k$ action). Restriction to $\Conf_k(X)$ is implemented by $\omega^*(-)$ and taking orbits is implemented by $p_!(-)$, that is:
\[
    ((E^{\overline{\wedge} k})|_{\Conf_k(X)})_{\Sigma_k} \simeq p_!(\omega^*E^{\overline{\wedge} k})
\]
in our notation. The Thom spectrum is computed using $r_!(-)$, hence Malkiewich's formula translated into our notation is
\[
    \big( ((E^{\overline{\wedge} k})|_{\Conf_k(X)})_{\Sigma_k} \big)^{-T\UConf_k(X)} \simeq r_!\big(p_!\omega^*(\Sigma^\infty_X E)^{\boxtimes k} \otimes \bS^{-T\UConf_k(X)}\big).
\]
Now $p^*\bS^{-T\UConf_k(X)} \simeq \bS^{-T\Conf_k(X)}$ so the projection formula gives
\begin{align*}
    r_!\big(p_!\omega^*(\Sigma^\infty_X E)^{\boxtimes k} \otimes \bS^{-T\UConf_k(X)}\big) &\simeq r_!p_!\big(\omega^*(\Sigma^\infty_X E)^{\boxtimes k} \otimes p^*\bS^{-T\UConf_k(X)}\big) \\
    &\simeq r_!p_!\big(\omega^*(\Sigma^\infty_X E)^{\boxtimes k} \otimes \bS^{-T\Conf_k(X)}\big).
\end{align*}
We observe furthermore that $\bS^{-T\Conf_k(X)} \simeq \omega^*\bS^{-TX^k}$ where $\bS^{-TX^k} \in \catSp^{X^k \sslash \Sigma_k}$ is seen with its $\Sigma_k$ action. Also $r \circ p = q \circ \omega$ is the unique map to a point, hence $r_!p_! = q_! \omega_!$. Combined with the projection formula this yields:
\begin{align*}
    r_!p_!\big(\omega^*(\Sigma^\infty_X E)^{\boxtimes k} \otimes \bS^{-T\Conf_k(X)}\big) &\simeq q_!\omega_!\big(\omega^*(\Sigma^\infty_X E)^{\boxtimes k} \otimes \omega^*\bS^{-TX^k}\big) \\
    &\simeq q_!\big(\Sigma^\infty_X E^{\boxtimes k} \otimes \bS^{-TX^k} \otimes \omega_!\bS^0\big).
\end{align*}
Finally, this last expression agrees with the desired formula: we use $q_!(-) \simeq (-)_{h(\Omega X \wr \Sigma_k)}$ recognising that $X^k \sslash \Sigma_k \simeq B(\Omega X \wr \Sigma_k)$.
\end{proof}

\begin{remark}
The spectrum $\omega_!\bS^0$ appearing in the formula of \cref{corollary:malkiewich-layers} sometimes goes under another name. It is indeed the suspension spectrum of the $\Omega X$-framed configuration space $\Conf_k^{\Omega X}(X)$, see e.g. \cite[Section~5.4]{HandbookPrimer}.
\end{remark}

\subsection{Recollections on Goodwillie calculus}

We briefly recall some results on Goodwillie calculus, mostly to fix our notation. We refer the unfamiliar reader to \cite{Handbook} for an introduction. Given a functor $F \colon \catS_*^X \to \catSp$ Goodwillie calculus provides a tower of functors
\[
    F \lra \cdots \lra P_kF \lra P_{k-1}F \lra \cdots \lra P_1F \lra P_0F.
\]
The zeroth approximation $P_0F$ is constant and each $P_kF$ is $k$-excisive in the sense of \cite[Definition~1.1.2]{Handbook}. The functor obtained as the fibre
\[
    D_kF = \mathrm{fib}(P_kF \lra P_{k-1}F)
\]
is called the $k$th layer of the Goodwillie tower. It is a $k$-homogeneous functor in the sense of \cite[Definition~1.2.2]{Handbook}. What makes the theory attractive and amenable to computations is the simple description of the layers.
\begin{proposition}\label{proposition:homogeneous-classification}
Let $F \colon \catS_*^X \to \catSp$ be a finitary functor, i.e. preserving all filtered colimits. There exists a spectrum $\partial_k F$ parameterised over $X^k \sslash \Sigma_k$ such that
\[
    D_kF(E) \simeq (\partial_kF \otimes \Sigma^\infty_X E^{\boxtimes k})_{h(\Omega X \wr \Sigma_k)}
\]
for all parameterised pointed spaces $E \in \catS_*^X$.
\end{proposition}
\begin{proof}
It follows from the general classification of homogeneous functors together with the fact that $\catSp^X \cong \catMod_{\bS[\Omega X]}(\catSp)$ is compactly generated by the single generator $\bS[\Omega X]$. See e.g. \cite[Section~1]{Ching}.
\end{proof}
In fact, a converse is true: any functor given by a formula
\[
    F(E) = (H \otimes \Sigma^\infty_X E^{\boxtimes k})_{h(\Omega X \wr \Sigma_k)}
\]
with some spectrum $H \in \catSp^{X^k \sslash \Sigma_k}$ is $k$-homogeneous. In what follows we shall use further results on Goodwillie calculus, which we will duly cite as needed. As a last piece of convention, we will assume that a generic functor $F$ is always assumed to be finitary.

\subsection{Calculus in two variables}

Goodwillie calculus is concerned with covariant functors from spaces, whereas Malkiewich's tower is for contravariant functors. We will explain how the towers agree in the very particular case of section spaces. Let
\begin{align*}
    \Gamma \colon (\catS^X)^\opp \times \catS_*^X &\lra \catSp \\
    (B,E) &\longmapsto \Sigma^\infty_+ \Gammacon(p^*E \to B), \ \text{ where } p \colon B \to X \text{ is the structure map}.
\end{align*}
\begin{proposition}[Malkiewich]\label{proposition:malkiewich-connectivity}
If $E \in \catS_*^X$ is $m$-connected, i.e. if the homotopy fibres of $E \to X$ are $(m-1)$-connected, then $\Gamma(-,E) \colon (\catS^X)^\opp \to \catSp$ is $m$-analytic with excess $0$ in Malkiewich's sense. In other words, if $B$ is a CW complex of dimension $d$ mapping to $X$ then the map
\[
    \Gamma(B,E) \lra A_k\Gamma(B,E)
\]
is $(k+1)(m-d)$-connected.
\end{proposition}
\begin{proof}
In Malkiewich's notation from \cite[p.~536]{Malkiewich} our section space is written $\Map_X(B \sqcup X, E) \cong \Gammacon(E)$. The relative dimension of $B \sqcup X$ defined in \cite[Definition~9.1]{Malkiewich} is the CW complex dimension of $B$ and the result follows by \cite[Proposition~9.3 and Theorem~9.4]{Malkiewich}.
\end{proof}

\begin{proposition}[Goodwillie]\label{proposition:goodwillie-connectivity}
There exists a constant $c \in \bZ$ such that the Goodwillie approximation map for the functor $\Gamma(X,-)$
\[
    \Gamma(X,E) \to P_k\Gamma(X,E)
\]
is $((k+1)(m - \dim X) - c)$-connected if $E \in \catS_*^X$ is $m$-connected (i.e. the homotopy fibres of $E \to X$ are $(m-1)$-connected).
\end{proposition}
\begin{proof}
See \cite[Definition~4.1]{GoodwillieII} for the definition of stable $k$th order excision and \cite[Propositions~1.4 and~1.5]{GoodwillieIII} for its usage in the proof of the result.
\end{proof}

\begin{lemma}\label{lemma:malkiewich-k-is-polynomial}
The functor $E \mapsto A_k\Gamma(X,E)$ is $k$-polynomial in the sense of Goodwillie.
\end{lemma}
\begin{proof}
We proceed by induction on $k$. For $k = 0$, $A_0\Gamma(X,E)$ is the $0$th layer in Malkiewich's tower which is seen to be constant. For $k > 0$, \cref{corollary:malkiewich-layers} shows that we have a fibre sequence of spectra
\[
    \big( \Sigma^\infty_X E^{\boxtimes k} \otimes \bS^{-TX^k} \otimes \omega_!\bS^0 \big)_{h(\Omega X \wr \Sigma_k)} \lra A_k\Gamma(X,E) \lra A_{k-1}\Gamma(X,E)
\]
which is natural in $E$ (by naturality of Malkiewich's tower). The left term is visibly $k$-homogeneous and the right term is $(k-1)$-polynomial by induction. Hence the middle term is $k$-polynomial.
\end{proof}

We are now ready to show that Malkiewich's and Goodwillie's towers agree.
\begin{proposition}
For any $E \in \catS_*^X$, $A_k\Gamma(X,E) \simeq P_k\Gamma(X,E)$.
\end{proposition}
\begin{proof}
By naturality of Malkiewich's tower, we can turn things around to get a covariant functor:
\[
    E \longmapsto A_k\Gamma(X,E).
\]
By \cref{lemma:malkiewich-k-is-polynomial} this functor is $k$-polynomial, so the universal property of $P_k\Gamma(X,-)$ provides a natural transformation
\[
    P_k\Gamma(X,-) \lra A_k\Gamma(X,-).
\]
By \cref{proposition:goodwillie-connectivity,,proposition:malkiewich-connectivity}, this map is an order $k$ agreement between functors. Hence their first $k$ derivatives agree. But they are also $k$-polynomial, and therefore must be isomorphic. 
\end{proof}

\subsection{From Goodwillie's to Weiss' tower}

Let $M \geq 0$ be big enough such that $S(\cL \otimes \bC^M)$ has a section. We now want to understand how Goodwillie calculus interacts with Weiss calculus for functors obtained by restriction along the functor
\begin{align*}
    r \colon \catJ &\lra \catS_*^X \\
    V &\longmapsto S(\cL \otimes (V \oplus \bC^M)).
\end{align*}
For a functor $F \colon \catS_*^X \to \catSp$, we write $r^*F$ for the unitary functor obtained by composing with $r$. We will explain how to obtain the Weiss tower of $r^*F$ from the Goodwillie tower of $F$. Let us start with two observations.

\begin{lemma}\label{lemma:orbit-tensor}
Let $\lambda \in U(1)$ act by scalar multiplication on $V$, hence on $S^V$, and by scalar multiplication by $\lambda^{-1}$ on $\cL$, hence on $S(\cL)$. Then there is a homotopy equivalence of pointed spaces over $X$
\[
    (S(\cL)_+ \wedge_X S^V)_{hU(1)} \cong S^{\cL \otimes V},
\]
where the homotopy orbits, the smash product, and the disjoint basepoint are taken fibrewise in pointed spaces, i.e. all constructions are interpreted in the category $\catS_*^X$, and $S^V$ is seen as the constant parameterised space.
\end{lemma}
\begin{proof}
We will construct an explicit equivalence. Given any $x \in X$, let $*_x \in S^{\cL \otimes V}$ denote the basepoint in the one-point compactification of $\cL|_x \otimes V$, write $(x,z) \in S(\cL)$ for a typical element in the fibre above $x$, and let $+_x \in S(\cL)_+$ denote the disjoint basepoint in the fibre above $x$. We define the map
\begin{align*}
    S(\cL)_+ \times S^V &\lra S^{\cL \otimes V} \\
    (x,z,v) &\longmapsto (x,z \otimes v) \text{ if } z \neq +_x\\
    (x,+_x,v) &\longmapsto *_x
\end{align*}
The subspace $S(\cL)_+ \times \{*_x, x \in X\} \cup \{+_x, x \in X\} \times S^V \subset S(\cL)_+ \times S^V$ is sent to $\{*_x, x \in X\} \subset S^{\cL \otimes V}$, so we obtain an induced map on the quotient
\[
    S(\cL)_+ \wedge_X S^V \lra S^{\cL \otimes V}.
\]
The $U(1)$ action is free away from the basepoints because it is free on $S(\cL)$. Thus homotopy orbits are computed by the standard point-set orbits. This latter map is $U(1)$-equivariant, where the target has the trivial action, so we obtain a map
\[
    (S(\cL)_+ \wedge_X S^V)_{U(1)} \lra S^{\cL \otimes V}.
\]
from the (homotopy) orbits. It is a homeomorphism: the bijectivity follows from the observation that $V \otimes \cL|_x \cong V$ because $\cL|_x$ is one-dimensional, and continuity of the inverse is clear.
\end{proof}

\begin{lemma}\label{lemma:untwisting}
There is a homeomorphism of pointed spaces over $X$
\[
    S^{\cL \otimes \bC^M} \wedge_X S(\cL)_+ \cong S^{\bC^M} \wedge_X S(\cL)_+.
\]
where all constructions are interpreted in $\catS_*^X$ and $S^{\bC^M}$ denotes the trivial sphere bundle over $X$.
\end{lemma}
\begin{proof}
With the same notations as in the proof of \cref{lemma:orbit-tensor}, a homeomorphism is given by
\begin{align*}
    S^{\cL \otimes \bC^M} \wedge_X S(\cL)_+ &\lra S^{\bC^M} \wedge_X S(\cL)_+ \\
    [x, z_x \otimes v, \lambda \cdot z_x] &\longmapsto [x, \lambda^{-1}\cdot v, \lambda \cdot z_x]
\end{align*}
where $z_x \in S(\cL)|_x$.
\end{proof}

We shall now explain how to compute the Weiss tower of $r^*F$ from the Goodwillie tower of $F$. The procedure is explained in \cite[Section~3]{BarnesEldred} in the unparameterised setting, that is for functors on $\catS_*$, and the extension to our parameterised setting is straightforward (and even slightly easier as we only work with spectrum valued functors). We will thus be brief with our proofs, mostly pointing out the differences. The strategy is clear: first consider homogeneous functors, then polynomial ones, and finally the whole tower. Recall that we write $\Ind_{(\Omega X \times U(1)) \wr \Sigma_k}^{U(k)}$ for the functor $(Bf)_!$ induced by $f \colon (\Omega X \times U(1)) \wr \Sigma_k \twoheadrightarrow U(1) \wr \Sigma_k \hookrightarrow U(k)$, i.e. the projection followed by the inclusion given by
\begin{equation}\label{equation:permutation-matrices}
    U(1) \wr \Sigma_k \ni (\lambda_i, \sigma) \longmapsto \text{diag}(\lambda_i) \cdot M(\sigma) \in U(k)
\end{equation}
where $M(\sigma)$ is the permutation matrix associated to $\sigma \in \Sigma_k$.

\begin{lemma}\label{lemma:goodwillie-weiss-homogeneous}
Let $F \colon \catS_*^X \to \catSp$ be a $k$-homogeneous functor in the sense of Goodwillie. Then $r^*F$ is $k$-homogeneous in the sense of Weiss. Furthermore, writing $\partial_k F$ for its Goodwillie derivative, the Weiss derivative is
\[
    \Ind_{(\Omega X \times U(1)) \wr \Sigma_k}^{U(k)} (\Sigma^{k(2M-1)} \partial_k F \otimes \Sigma^\infty_X S(\cL)_+^{\boxtimes k})
\]
where the actions are given by
\begin{itemize}
    \item $\Omega X$ acts on $\partial_kF$ and on $\Sigma^\infty_X S(\cL)_+$ by definition of parameterised spectra $\catSp^X \simeq \catSp^{B\Omega X}$,
    \item $\lambda \in U(1)$ acts on $S(\cL)$ by multiplication by $\lambda^{-1}$, and trivially on $\partial_k F$,
    \item the symmetric group $\Sigma_k$ acts by definition on $\partial_kF$, and by permutation on the factors of $\Sigma^\infty_X S(\cL)_+^{\boxtimes k}$ and of $\Sigma^{k(2M-1)} = (\Sigma^{2M-1})^{\otimes k}$.
\end{itemize}
\end{lemma}
\begin{proof}
It is a direct computation from the general form of a homogeneous functor recalled in \cref{proposition:homogeneous-classification}. The functor $F$ is given by the formula
\[
    F(E) = (\partial_kF \otimes \Sigma^\infty_X E^{\boxtimes k})_{h(\Omega X \wr \Sigma_k)}
\]
Therefore
\[
    r^*F(V) = (\partial_kF \otimes (\Sigma^\infty_X S(\cL \otimes (V \oplus \bC^M)))^{\boxtimes k})_{h(\Omega X \wr \Sigma_k)}.
\]
There is a cofibre sequence of pointed spaces over $X$:
\[
    S(\cL \otimes (V \oplus \bC^M)) \lra D(\cL \otimes (V \oplus \bC^M)) \lra S^{\cL \otimes (V \oplus \bC^M)},
\]
where the sphere and disc bundles are pointed using the fixed section of $S(\bC^M \otimes \cL)$. By suspending we get
\[
    \Sigma^\infty_X S(\cL \otimes (V \oplus \bC^M)) \simeq \Sigma^{-1} \bS^{\cL \otimes (V \oplus \bC^M)} \simeq \Sigma^{-1} (\bS^{\cL \otimes V} \otimes \bS^{\cL \otimes \bC^M}).
\]
Plugging this into the formula for $r^*F$ yields
\[
    r^*F(V) \simeq (\partial_k F \otimes \Sigma^{-k}(\bS^{\cL \otimes V})^{\boxtimes k} \otimes (\bS^{\cL \otimes \bC^M})^{\boxtimes k})_{h(\Omega X \wr \Sigma_k)}.
\]
By \cref{lemma:orbit-tensor} we have an equivalence
\[
    \bS^{\cL \otimes V} \simeq (\Sigma^\infty_X S(\cL)_+ \otimes \bS^{V})_{hU(1)},
\]
which we plug into the formula of $r^*F(V)$ to obtain
\begin{align*}
    r^*F(V) &\simeq (\partial_k F \otimes \Sigma^{-k}(\bS^{\cL \otimes \bC^M})^{\boxtimes k} \otimes ((\Sigma^\infty_X S(\cL)_+ \otimes \bS^{V})_{hU(1)})^{\boxtimes k})_{h(\Omega X \wr \Sigma_k)} \\
    &\simeq (\partial_k F \otimes\Sigma^{-k} (\bS^{\cL \otimes \bC^M})^{\boxtimes k} \otimes \Sigma^\infty_X S(\cL)_+^{\boxtimes k} \otimes \bS^{kV})_{h(\Omega X \times U(1)) \wr \Sigma_k}.
\end{align*}
By \cref{lemma:untwisting} there is an equivalence
\[
    \bS^{\cL \otimes \bC^M} \otimes \Sigma^\infty_X S(\cL)_+ \simeq \bS^{2M} \otimes \Sigma^\infty_X S(\cL)_+
\]
where $\bS^{2M}$ is the constant parameterised spectrum. Therefore
\[
    r^*F(V) \simeq \Ind_{(\Omega X \times U(1)) \wr \Sigma_k}^{U(k)} (\Sigma^{k(2M-1)} \partial_k F \otimes \Sigma^\infty_X S(\cL)_+^{\boxtimes k}) \otimes_{U(k)} \bS^{kV}
\]
from which we can read off the Weiss derivative as desired.
\end{proof}

\begin{lemma}\label{lemma:goodwillie-weiss-polynomial}
Let $F \colon \catS_*^X \to \catSp$ be a $k$-polynomial functor in the sense of Goodwillie. Then $r^*F$ is $k$-polynomial in the sense of Weiss.
\end{lemma}
\begin{proof}
The proof is by induction on $k$ as in \cite[Proposition~3.2]{BarnesEldred}. The case $k = 0$ is clear. For $k > 0$ there is a fibre sequence
\[
    D_kF \lra P_kF \lra P_{k-1}F
\]
where $D_kF$ is $k$-homogeneous. Since fibre sequences are defined pointwise in functor categories, the restricted sequence
\[
    r^*(D_kF) \lra r^*(P_kF) \lra r^*(P_{k-1}F)
\]
is also a fibre sequence. By \cref{lemma:goodwillie-weiss-homogeneous} the fibre is $k$-homogeneous, and by induction $r^*(P_{k-1}F)$ is $(k-1)$-polynomial. Hence $r^*(P_kF)$ is $k$-polynomial.
\end{proof}

\begin{lemma}\label{lemma:goodwillie-weiss-agreement}
If $F \colon \catS_*^X \to \catSp$ is stably $k$-excisive in the sense of Goodwillie \cite[Definition~4.1]{GoodwillieII}, then $r^*F \to r^*P_kF$ is a $k$th order agreement in the sense of Weiss.
\end{lemma}
\begin{proof}
This is the parameterised version of \cite[Lemma~2.29]{BarnesEldred}. The only difference is that connectivity has to be interpreted in the parameterised sense. In particular $S(\cL \otimes (V \oplus \bC^M))$ is $(\dim(V) + 2M-2)$-connected as a parameterised space over $X$.
\end{proof}

\begin{lemma}\label{lemma:goodwillie-weiss-approximation}
If $F$ is stably $k$-excisive, then $T_k (r^*F) \to T_k(r^* P_kF) \simeq r^* P_kF$ is an equivalence. Thus the $k$-polynomial approximation of $r^*F$ is given by the map $r^*F \to r^*P_kF$.
\end{lemma}
\begin{proof}
This can be proved as in \cite[Proposition~3.4]{BarnesEldred}. There is a commutative diagram
\[
\begin{tikzcd}
r^*F(V) \arrow[d] \arrow[r]      & r^*(P_kF)(V) \arrow[d, "\simeq"] \\
T_k(r^*F)(V) \arrow[r, "\simeq"] & T_k(r^*P_kF)(V)                 
\end{tikzcd}
\]
The right vertical arrow is an equivalence by \cref{lemma:goodwillie-weiss-polynomial}. The bottom horizontal arrow is an equivalence as a consequence of \cref{lemma:goodwillie-weiss-agreement}.
\end{proof}

\begin{proposition}\label{proposition:goodwillie-weiss-tower}
If $F$ is an analytic functor then so is $r^*F$. Furthermore the Weiss tower of $r^*F$ is the restriction of the Goodwillie tower of $F$.
\end{proposition}
\begin{proof}
This follows from \cref{lemma:goodwillie-weiss-approximation} by induction on $k$. See \cite[Theorem~3.5]{BarnesEldred} for more details.
\end{proof}

\subsection{Conclusion for continuous maps}

Let $\cL$ be a complex line bundle on $X$ and let $\alpha = c_1(\cL)$. In this subsection, we will furthermore assume that $X$ is a finite dimensional finite CW complex. Let us also pick some $M > \frac{1}{2}\dim_{\mathrm{CW}}(X)$ so that $S(\cL \otimes \bC^M)$ has a section. The group $\Map(X,U(1))$ acts by bundle automorphisms on $\cL$: for $g \in \Map(X,U(1))$ and $z \in \cL|_x$ the action is given by the scalar multiplication $g \cdot z = g(x)z$. In fact, by \cite[Proposition~2.6]{CrabbSutherland} there is a homeomorphism
\begin{equation}\label{equation:principal-map-bundle}
    \Map_\alpha(X, \bP(\bC^M \oplus V)) \cong \Gammacon(S(\cL \otimes (\bC^M \oplus V))) / \Map(X,U(1))
\end{equation}
where the orbit map is a principal $\Map(X,U(1))$-bundle. Using what we have done so far, we can thus precisely state and prove \cref{maintheorem:continuous} using the notation from \cref{theorem:malkiewich-tower}.
\begin{theorem}\label{theorem:final-continuous-answer}
Let $X$ be a finite dimensional finite CW complex. Let $\cL$ be a complex line bundle on $X$ with first Chern class $\alpha = c_1(\cL)$, and let $M > \frac{1}{2}\dim_{\mathrm{CW}}(X)$. Then the unitary functor 
\[
    V \longmapsto \Sigma^\infty_+ \Map_\alpha(X, \bP(\bC^M \oplus V))
\]
has a converging Weiss tower. Furthermore, its $k$th approximation in the Weiss tower is given by the formula
\[
    \Map_{(\catEpi_k^\opp, \{X^\bullet\})}(X^{\overline{\wedge} \bullet}, \Sigma^\infty_{X^\bullet}S(\cL \otimes (\bC^M \oplus V))^{\overline{\wedge} \bullet}) \sslash \Map(X,U(1))
\]
where the action of $\Map(X,U(1))$ is induced from the automorphism action on $\cL$. In particular, the $k$th derivative is given by
\[
    \Theta_k \simeq \left(\Ind_{U(1) \wr \Sigma_k}^{U(k)} \Sigma^{k(2M-1)}  \Conf_k(X, S(\cL))^{-kTX}\right) \sslash \Map(X,U(1))
\]
where $\Conf_k(X, S(\cL)) = S(\cL)^{\times k}|_{\Conf_k(X)}$ is the configuration space of $k$ points on $X$ with labels in the sphere bundle $S(\cL)$. The actions are given by
\begin{itemize}
    \item an element $\lambda \in U(1)$ acts by scalar multiplication by $\lambda^{-1}$ on $S(\cL)$,
    \item the symmetric group $\Sigma_k$ acts by permutation on the configurations, on the factors of $kTX = TX \oplus \cdots \oplus TX$, and on the suspensions factors $\Sigma^{k(2M-1)} = (\Sigma^{2M-1})^{\otimes k}$.
\end{itemize}
The group $U(1) \wr \Sigma_k = U(1)^k \rtimes \Sigma_k$ sits inside $U(k)$ via the permutation matrices (see the formula~\eqref{equation:permutation-matrices}), and $\Ind_{U(1) \wr \Sigma_k}^{U(k)}$ denotes the induced action. In the final homotopy quotient, $\Map(X,U(1))$ acts by automorphisms on $S(\cL)$.
\end{theorem}
\begin{proof}
The orbits in \cref{equation:principal-map-bundle} are homotopy orbits as the action is free. Hence
\[
    \Sigma^\infty_+ \Map_\alpha(X, \bP(\bC^M \oplus V)) \simeq \big(\Sigma^\infty_+\Gammacon(S(\cL \otimes (\bC^M \oplus V))) \big) \sslash \Map(X,U(1)).
\]
The result then follows by using \cref{lemma:parameterised-colimit-tower,,theorem:malkiewich-tower,,proposition:goodwillie-weiss-tower} for the tower, and \cref{lemma:parameterised-colim-derivatives,,corollary:malkiewich-layers,,lemma:goodwillie-weiss-homogeneous} and the equivalence, obtained by applying the projection formula,
\[
    \big(\bS^{-kTX} \otimes \omega_!\bS^0 \otimes \Sigma^\infty_X S(\cL)_+^{\boxtimes k} \big)_{h\Omega X^k} \simeq \Conf_k(X, S(\cL))^{-kTX}
\]
for the derivatives. Convergence of the tower is immediate from convergence of the Goodwillie tower, which follows from convergence of the Malkiewich tower of the section space, see \cite[Example~4.1]{Malkiewich}.
\end{proof}

\subsection{Conclusion for holomorphic maps}\label{subsection:conclusion}

The first half of \cref{maintheorem:holomorphic} was proven as \cref{theorem:final-holomorphic-polynomiality}. Before proving the second half, let us recall the following definitions for the sake of completeness.
\begin{definition}
Let $\cL$ be a line bundle on a smooth complex projective variety $X$ and $k \geq 0$ be an integer. We say that $\cL$ \emph{separates $k$ points} if for any choice of $k+1$ distinct points $x_0,\dotsc,x_k \in X$ the evaluation map
\[
    \Gammahol(\cL) \lra \bigoplus_{i=0}^k \cL|_{x_i}, \quad s \longmapsto (s(x_0),\dotsc,s(x_k))
\]
is surjective.
\end{definition}
\begin{definition}
Let $X$ be a smooth complex projective variety, and let $\alpha \in \NS(X)$ be class in the Néron--Severi group of $X$. We define $d(\alpha) \in \bN \cup \{-\infty\}$ as
\[
    d(\alpha) = \sup \ \set{k \in \bN}{\forall [\cL] \in \Pic^\alpha(X), \ \cL \text{ separates $k$ points}}.
\]
When the set of the right-hand side is empty $d(\alpha) = - \infty$. By convention we also set $d(\alpha) = - \infty$ if $\alpha \notin \NS(X)$.
\end{definition}
For more details about this number $d(\alpha)$, we refer the reader to \cite[Section~3]{Aumonier}. The main result of \cite{Aumonier} shows that the inclusion
\[
    \Hol_\alpha(X,\bP(\bC^{n+1} \oplus V)) \subset \Map_\alpha(X,\bP(\bC^{n+1} \oplus V))
\]
induces an isomorphism in integral homology in the range of degrees $* < 2(d(\alpha)+1)\dim V +d(\alpha)-2$. Hence, by \cref{proposition:numerical-observation}, the second half of \cref{maintheorem:holomorphic} follows from \cref{theorem:final-continuous-answer}.

\section{Application: a stable splitting for rational maps}\label{section:splitting}

In this section we show how the stable splitting of Cohen--Cohen--Mann--Milgram \cite{CCMM}, obtained via homological computations by the authors, can be proven more conceptually using unitary calculus. We will consider pointed holomorphic maps, so we introduce some notation. Let $\infty \in \bC \cup \{\infty\} = \bP^1$ be the basepoint of the Riemann sphere. For any $V \in \catJ$, we choose the line $[\bC \oplus 0] \in \bP(\bC^2 \oplus V)$ as a basepoint. The space of \emph{rational maps} of degree $d \in \bZ \cong H^2(\bP^1;\bZ)$ is the space
\[
    \Rat_d(V) = \Hol_d^*(\bP^1, \bP(\bC^2 \oplus V))
\]
of pointed holomorphic maps of degree $d$. We regard it as a unitary functor in $V$. We first prove a refined version of \cref{theorem:final-holomorphic-polynomiality} in the special case of pointed maps.
\begin{lemma}\label{lemma:rat-polynomiality}
The unitary functor
\[
    V \longmapsto \Sigma^\infty_+\Rat_d(V)
\]
is $d$-polynomial.
\end{lemma}
\begin{remark}
Without the basepoint condition on the maps, the functor is $(d+1)$-polynomial as shown in \cref{theorem:final-holomorphic-polynomiality}. The point here is that we can do slightly better.
\end{remark}
\begin{proof}
Any holomorphic map $\bP^1 \to \bP(\bC^2 \oplus V)$ of degree $d$ is given by a tuple of polynomials $(f_0, \ldots, f_{\dim V +1})$ homogeneous in $x,y$ of degree $d$ and defined up to a non-zero complex scalar. It is pointed if
\[
    [f_0(1,0) : \ldots : f_{\dim V+1}(1,0)] = [1:0:\ldots:0] \in \bP(\bC^2 \oplus V).
\]
For a such a map, we can remove the scalar ambiguity in the definition of the $f_i$ by normalising and dehomogenising the polynomials. More precisely, we can assume that they are polynomials in one variable $x$ with $f_0$ monic of degree $d$ and the other $f_i$ of degree $\leq d-1$. In other words, we can write a rational map $f \colon \bP^1 \to \bP(\bC^2 \oplus V)$ as a tuple of polynomials $(f_0,\dotsc,f_{\dim V +1})$ of the form
\[
    f_0 = x^d + a_{d-1}^{(0)}x^{d-1} + \ldots + a_0^{(0)}
\]
and for $1 \leq i \leq \dim V +1$
\[
    f_i = a_{d-1}^{(i)}x^{d-1} + \ldots + a_0^{(i)}
\]
where the $a_j^{(i)} \in \bC$ are coefficients, and the $f_i$ do not vanish simultaneously. More functorially
\[
    \Rat_d(V) = \set{A \in \Hom(\bC^d, \bC^2 \oplus V)}{\bP(\ker(\mathrm{id}_\bC +A)) \cap \nu_d(\bP^1) = \emptyset}
\]
where $\nu_d \colon \bP^1 \hookrightarrow \bP^d$ is the $d$th Veronese embedding. In that identification, the linear map $A$ is in matrix form:
\[
A = \begin{pmatrix}
        a_{d-1}^{(0)} & \ldots & a_0^{(0)} \\
        a_{d-1}^{(1)} & \ldots & a_0^{(1)} \\
        \vdots & \ldots & \vdots \\
        a_{d-1}^{(\dim V+1)} & \ldots & a_0^{(\dim V +1)} \\
    \end{pmatrix}
\]
The argument in the proof of \cref{theorem:holomorphic-polynomiality} will show that the functor if $d$-polynomial: the crucial observation is that we have linear maps from $\bC^d$ instead of $\bC^{d+1}$.
\end{proof}

The following is the promised application.
\begin{theorem}[Cohen--Cohen--Mann--Milgram {\cite[Theorem~1.5]{CCMM}}]\label{theorem:unitary-ccmm}
The space of rational maps stably splits:
\[
    \Sigma^\infty_+ \Rat_d(V) \simeq \bigoplus_{i=0}^d \Sigma^{i\dim V}\left(\Conf_i(\bC)_+ \wedge (S^1)^{\wedge i}\right)_{\Sigma_i}
\]
\end{theorem}
\begin{proof}
The Snaith splitting \cite{Cohen} is the equivalence
\[
    \Sigma^\infty_+ \Omega^2_d\bP(\bC^2 \oplus V) \simeq \bigoplus_{i=0}^\infty \Sigma^{i\dim V}\left(\Conf_i(\bC)_+ \wedge (S^1)^{\wedge i}\right)_{\Sigma_i}.
\]
Each summand is a homogeneous functor in the sense of unitary calculus. Hence we can interpret this result as saying that the Weiss tower of $\Sigma^\infty_+ \Omega^2_d\bP(\bC^2 \oplus V)$ splits. By \cite{Aumonier}, the map
\[
    \Rat_d(V) \hookrightarrow \Map_d^*(\bP^1, \bP(\bC^2 \oplus V)) = \Omega^2_d\bP(\bC^2 \oplus V)
\]
induces an isomorphism on homology in the range $* < 2(d+1)\dim_\bC(V) -d-2$. Hence by \cref{proposition:numerical-observation} the suspension spectra have the same $d$th polynomial approximation. The result follows as $\Sigma^\infty_+ \Rat_d(V)$ is already $d$-polynomial by \cref{lemma:rat-polynomiality}.
\end{proof}

\bigskip

To showcase another technique, we will now show how to recover the splitting of the Weiss tower of $\Sigma^\infty_+\Rat_d(V)$ without using the Snaith splitting. By recording the roots of the polynomials defining a rational map, Segal \cite[Section~4, p.~49]{Segal} has constructed stabilisations maps, essentially by gluing configurations of roots\footnote{More precisely, the disjoint union $\bigsqcup_{d \geq 0} \Rat_d(V)$ is given the structure of an algebra over the little 2-discs operad.},
\[
    \Rat_1(V) \lra \Rat_2(V) \lra \Rat_3(V) \lra \cdots
\]
which are functorial in $V \in \catJ$. This is how they interact with the Weiss towers:
\begin{lemma}\label{lemma:segal-stabilisation}
The stabilisation map
\[
    \Sigma^\infty_+\Rat_d(V) \lra \Sigma^\infty_+\Rat_{d+1}(V)
\]
induces an isomorphism between the $d$th polynomial approximations.
\end{lemma}
\begin{proof}
We use the inclusion
\[
    \Rat_d(V) \subset \Map_d^*(\bP^1, \bP(\bC^2 \oplus V)) = \Omega^2_d\bP(\bC^2 \oplus V).
\]
The loop space multiplication $\Omega^2_d \bP(\bC^2 \oplus V) \to \Omega^2_{d+1}\bP(\bC^2 \oplus V)$ with the fixed element in $\Rat_1(V) \subset \Omega^2_1\bP(\bC^2 \oplus V)$ corresponding to the stabilisation map fits into a homotopy commutative diagram:
\[
\begin{tikzcd}
\Rat_d(V) \arrow[d, hook] \arrow[r] & \Rat_{d+1}(V) \arrow[d, hook] \\
\Omega^2_d \bP(\bC^2 \oplus V) \arrow[r]   & \Omega^2_{d+1} \bP(\bC^2 \oplus V) 
\end{tikzcd}
\]
The bottom map is an isomorphism in homology (all components of the loop space are homotopy equivalent). The vertical maps induce isomorphisms in homology in the range at least $* < 2(d+1)\dim_\bC(V) -d-2$ by \cite{Aumonier}. Hence the stabilisation map at the top is also an isomorphism in homology in that range. Therefore it induces an isomorphism between the $d$th polynomial approximations by \cref{proposition:numerical-observation}.
\end{proof}

\begin{remark}
The lemma nearly follows from Segal's work. However the homological bound in his theorem is only optimal for maps to $\bP^1$ (this is a consequence of \cite{CCMM} which was observed then, but also of \cite{Aumonier}). Combined with \cref{proposition:numerical-observation} we would only be able to prove an isomorphism between the $(d-1)$st polynomial approximations.
\end{remark}

We can now prove again the stable splitting.
\begin{theorem}
The Weiss tower of $\Sigma^\infty_+ \Rat_d(V)$ splits.
\end{theorem}
\begin{proof}
The argument we use already appears in \cite[bottom of page~1210]{Arone} in the case of Miller's splitting. The stabilisation map gives a filtration:
\[
    \Sigma^\infty_+ \Rat_0(V) \lra \Sigma^\infty_+ \Rat_1(V) \lra \ldots \lra \Sigma^\infty_+ \Rat_d(V).
\]
\Cref{lemma:segal-stabilisation,,lemma:rat-polynomiality} show that the associated graded pieces are homogeneous functors. By taking polynomial approximations, we get a diagram where the rows are cofibre sequences:
\[
\begin{tikzcd}
\Sigma^\infty_+ \Rat_i(V) \arrow[d, "\simeq"'] \arrow[r] & \Sigma^\infty_+ \Rat_{i+1}(V) \arrow[d] \arrow[r] & \Sigma^\infty_+ \Rat_{i+1}(V) / \Sigma^\infty_+ \Rat_i(V) \arrow[d]     \\
T_i \Sigma^\infty_+ \Rat_i(V) \arrow[r, "\simeq"]        & T_i\Sigma^\infty_+ \Rat_{i+1}(V) \arrow[r]        & T_i(\Sigma^\infty_+ \Rat_{i+1}(V) / \Sigma^\infty_+ \Rat_i(V)) \simeq 0
\end{tikzcd}
\]
The bottom right functor is zero as the cofibre is $(i+1)$-homogeneous. Hence the bottom left map is an equivalence. The left vertical morphism is an equivalence by \cref{lemma:rat-polynomiality}. Hence $\Sigma^\infty_+ \Rat_i(V)$ is a retract of $\Sigma^\infty_+ \Rat_{i+1}(V)$ and the result follows by induction.
\end{proof}

\section{Case study}\label{section:case-study}

In this section, we have compiled some computations of derivatives of unitary functors of holomorphic maps. Our goal is to showcase how the approach developed in this paper looks like in concrete examples.

\subsection{Degree \texorpdfstring{$1$}{1} maps from \texorpdfstring{$\bP^n$}{Pn}}\label{subsection:pn-case}

Let us start with the easiest case in the theory: the unitary functor
\[
    V \longmapsto \Sigma^\infty_+ \Hol_1(\bP^n, \bP(\bC^{n+1} \oplus V))
\]
of holomorphic maps of degree $1$. There is a homeomorphism
\[
    \Hol_1(\bP^n, \bP(\bC^{n+1} \oplus V)) \cong \Gammahol(\cO_{\bP^n}(1) \otimes (\bC^{n+1} \oplus V) - 0) / \bC^\times,
\]
so, in light of \cref{subsection:polynomiality-full-mapping-space}, we will in fact work with the section space in the remainder of this section. We start with an observation.
\begin{lemma}
There is a homotopy equivalence
\[
    \Gammahol(\cO_{\bP^n}(1) \otimes (\bC^{n+1} \oplus V) - 0) \simeq U(\bC^{n+1} \oplus V) / U(V)
\]
with the Stiefel manifold of unitary $(n+1)$-frames in $\bC^{n+1} \oplus V$.
\end{lemma}
\begin{proof}
Under the identification
\[
    \Gammahol(\cO_{\bP^n}(1) \otimes (\bC^{n+1} \oplus V)) \cong \Hom(\bC^{n+1}, \bC^{n+1} \oplus V),
\]
the never-vanishing sections corresponds to the linear maps of full rank. By the Gram--Schmidt process, this subspace is homotopy equivalent to the Stiefel manifold.
\end{proof}
By Miller's splitting \cite{Arone}, the Weiss derivatives of
\[
    V \longmapsto \Sigma^\infty_+ U(\bC^{n+1} \oplus V) / U(V)
\]
are given by
\[
    \Theta_k = \bS^{\mathrm{ad}_k} \otimes \Sigma^\infty_+ \Hom_\catJ(\bC^k, \bC^{n+1})
\]
where $\mathrm{ad}_k$ is the adjoint representation of $U(k)$, and $A \in U(k)$ acts on $f \in \Hom_\catJ(\bC^k,\bC^{n+1})$ via its inverse $(A \cdot f)(x) = f(A^{-1}x)$. 

\subsubsection{The first derivative}
The first derivative has to agree with the one given in \cref{maintheorem:holomorphic} (without the quotient by the group $\Map(\bP^n,U(1)) \simeq \bC^\times$ as we are only looking at the section space). Let us check this for good measure:
\begin{lemma}
There is an equivalence of $U(1)$-spectra
\[
    \bS^{\mathrm{ad}_1} \otimes \Sigma^\infty_+ \Hom_\catJ(\bC^1, \bC^{n+1}) \simeq \Sigma^{2n+1} S(\cO_{\bP^n}(1))^{-T\bP^n}.
\]
The action of $U(1)$ on the left term is as above, and recall that the action of $\lambda \in U(1)$ on the right term is by acting via scalar multiplication by $\lambda^{-1}$ on $S(\cO_{\bP^n}(1))$.
\end{lemma}
\begin{proof}
An isometric embedding $\bC \to \bC^{n+1}$ is completely determined by the image of $1 \in \bC$, hence the homeomorphism of $U(1)$-spaces
\[
    \Hom_\catJ(\bC^1, \bC^{n+1}) \cong S(\bC^{n+1})
\]
where $\lambda \in U(1)$ acts by multiplication by $\lambda^{-1} = \overline{\lambda}$. Furthermore, the adjoint representation is trivial so $\bS^{\mathrm{ad}_1} = \bS^1$. We also recognise that $S(\bC^{n+1})$ is the total space of the sphere bundle of the tautological bundle $\cO_{\bP^n}(1)$. We thus have
\[
    \bS^{\mathrm{ad}_1} \otimes \Sigma^\infty_+ \Hom_\catJ(\bC^1, \bC^{n+1}) \simeq \Sigma^1 \Sigma^\infty_+S(\cO_{\bP^n}(1))
\]
so far. Finally, we use the isomorphism of bundles $T\bP^n \oplus \underline{\bC} \cong \cO_{\bP^n}(1)^{\oplus n+1}$ on $\bP^n$. When pulled back to the sphere bundle $S(\cO_{\bP^n}(1))$, the line bundle $\cO_{\bP^n}(1)$ is canonically trivialised. Hence $-T\bP^n = -n\underline{\bC}$ as virtual bundles on $S(\cO_{\bP^n}(1))$. Therefore
\[
    \Sigma^1 \Sigma^\infty_+S(\cO_{\bP^n}(1)) \simeq \Sigma^{2n+1} \Sigma^\infty S(\cO_{\bP^n}(1))^{-n\underline{\bC}} \simeq \Sigma^{2n+1} S(\cO_{\bP^n}(1))^{-T\bP^n}. \qedhere
\]
\end{proof}

\subsubsection{The second derivative}
The second derivative $\Theta_2$ is unstable in the sense that it has a priori no reason to agree with the one given in \cref{maintheorem:continuous}, namely
\[
    \widetilde{\Theta}_2 = \Ind_{U(1) \wr \Sigma_2}^{U(2)} \Sigma^{2(2n+1)}  \Conf_2(\bP^n, S(\cO_{\bP^n}(1)))^{-2T\bP^n}.
\]
Let us indeed show that they are different. For concreteness, we take $n = 1$ and look at the second layers of the Weiss towers of
\begin{equation}\label{equation:secondderivative-hol}
    V \longmapsto \Sigma^\infty_+\Gammahol(\cO_{\bP^1}(1) \otimes (\bC^2 \oplus V)-0)
\end{equation}
and
\begin{equation}\label{equation:secondderivative-con}
    V \longmapsto \Sigma^\infty_+\Gammacon(\cO_{\bP^1}(1) \otimes (\bC^2 \oplus V)-0)
\end{equation}
The former is the spectrum
\[
    \bS^{\mathrm{ad}_2} \otimes \Sigma^\infty_+ \Hom_\catJ(\bC^2, \bC^2) \otimes_{U(2)} \bS^{2V} \simeq \bS^{\mathrm{ad}_2} \otimes \Sigma^\infty_+ U(2) \otimes_{U(2)} \bS^{2V} \simeq \bS^4 \otimes \bS^{2V},
\]
using that $\ad_2$ is a real $4$-dimensional representation. The latter is the spectrum
\[
    \widetilde{\Theta}_2 \otimes_{U(2)} \bS^{2V} \simeq \Sigma^3\Sigma^3  \Conf_2(\bP^1, S(\cO_{\bP^1}(1)))^{-2T\bP^1} \otimes_{U(1) \wr \Sigma_2} \bS^{2V},
\]
where we have written $\Sigma^3\Sigma^3$ instead of $\Sigma^6$ to remember that $\Sigma_2$ acts by swapping the suspensions. We will show that for $V = 0$, the integral homologies of these two spectra are different in degree $2$. We compute
\begin{align*}
    (\widetilde{\Theta}_2 \otimes_{U(2)} \bS^{0}) \otimes \bZ &\simeq \left(\Sigma^3\Sigma^3  \Conf_2(\bP^1, S(\cO_{\bP^1}(1)))^{-2T\bP^1} \right)_{h(U(1) \wr \Sigma_2)} \otimes \bZ \\
    &\simeq \Sigma^2 \left(\bZ^{\mathrm{\pm}} \otimes  \Conf_2(\bP^1, S(\cO_{\bP^1}(1)))_+ \right)_{h(U(1) \wr \Sigma_2)},
\end{align*}
where $\bZ^{\pm}$ is the sign representation, and we have used the Thom isomorphism. Thus
\begin{align*}
    H_2\big(\widetilde{\Theta}_2 \otimes_{U(2)} \bS^{0}; \bZ\big) &\cong H_0\big((\Conf_2(\bP^1, S(\cO_{\bP^1}(1)))_+)_{h(U(1) \wr \Sigma_2)};\bZ^\pm\big) \\
    &\cong H_0\big(\UConf_2(\bP^1);\bZ^\pm\big) \\
    &\cong H_0\big(B\Sigma_2; \bZ^\pm\big)\\
    &\cong \bZ/2.
\end{align*}
On the other hand,
\[
    H_2(\bS^4 \otimes \bS^{0}; \bZ) = 0.
\]

\subsection{Degree \texorpdfstring{$(1,1)$}{(1,1)} maps from \texorpdfstring{$\bP^1\times\bP^1$}{P1xP1}}\label{subsection:p1p1-case}
We now study the space of degree $(1,1)$ holomorphic maps from a product of projective lines $\bP^1\times\bP^1$, or more precisely the unitary functor
\begin{equation}\label{equation:functor-sections-p1p1}
    V \longmapsto \Sigma^\infty_+ \Gammahol(\cO_{\bP^1\times\bP^1}(1,1) \otimes (\bC^4 \oplus V) - 0).
\end{equation}
This is the simplest case which exemplifies most of the key ideas of \cref{section:polynomiality}, so we spend some time to flesh them out concretely. Let $[x:y]$ and $[w:z]$ be the homogeneous coordinates on the two copies of $\bP^1$. A global holomorphic section of $\cO_{\bP^1\times\bP^1}(1,1)$ is a homogeneous polynomial
\[
    f = a xw + b xz + c yw + d yz
\]
with $a,b,c,d \in \bC$. We may thus represent a section of $\cO_{\bP^1\times\bP^1}(1,1) \otimes (\bC^4 \oplus V)$ as a tuple $(f_1,\dotsc,f_{v+4})$ of homogeneous polynomials, with $v = \dim V$. We collect their coefficients in a matrix
\[
    \begin{pmatrix}
        a_1 & \ldots & a_{v+4} \\
        b_1 & \ldots & b_{v+4} \\
        c_1 & \ldots & c_{v+4} \\
        d_1 & \ldots & d_{v+4} \\
    \end{pmatrix}
\]
which we like to regard as a linear map $\bC^4 \lra \bC^4 \oplus V$. Observe that such a linear map corresponds to a nowhere vanishing section if and only if its kernel $K \subset \bC^4$ is such that $\bP(K) \cap Q = \emptyset$, where $Q \subset \bP^3$ is the Segre quadric, i.e. the image of the embedding $\bP^1 \times \bP^1 \hookrightarrow \bP^3$ given by the complete linear system $|\cO_{\bP^1\times\bP^1}(1,1)|$. As $Q \subset \bP^3$ has codimension $1$, any such $K$ is forced to have dimension either $0$ or $1$. We thus obtain a stratification
\begin{align*}
    \Gammahol(\cO_{\bP^1\times\bP^1}(1,1) \otimes (\bC^4 \oplus V) - 0) &\cong \set{A \in \Hom(\bC^4, \bC^4 \oplus V)}{\bP(\ker A) \cap Q = \emptyset} \\
    &= \set{A}{\ker A = 0} \sqcup \set{A}{\dim \ker A = 1 \text{ and } \bP(\ker A) \not\in Q}.
\end{align*}
The first piece is open inside the total space, and the second is a closed submanifold. Let us denote by $\nu$ its normal bundle, and by $S(\nu)$ and $D(\nu)$ the associated sphere and disc bundles. We have a homotopy pushout square of spaces
\[
\begin{tikzcd}
S(\nu) \arrow[r, hook] \arrow[d]    & D(\nu) \simeq \set{A}{\dim \ker A = 1 \text{ and } \bP(\ker A) \not\in Q} \arrow[d]    \\
\set{A}{\ker A = 0} \arrow[r, hook] & {\set{A \in \Hom(\bC^4, \bC^4 \oplus V)}{\bP(\ker A) \cap Q = \emptyset}}
\end{tikzcd}
\]
expressing that the total space is obtained from the open piece by gluing in a tubular neighbourhood of the closed submanifold. After applying $\Sigma^\infty_+$ to every space in sight, the Weiss derivatives of the pushout can be computed as the pushout of the derivatives. The bottom left corner is the functor
\[
    V \longmapsto \Sigma^\infty_+\set{A \in \Hom(\bC^4, \bC^4 \oplus V)}{\ker A} \simeq \Sigma^\infty_+ U(\bC^4\oplus V)/U(V)
\]
so we know its Weiss tower by Miller's splitting. The top right corner fibres via the kernel map
\[
    \set{A \in \Hom(\bC^4, \bC^4 \oplus V)}{\dim \ker A = 1 \text{ and } \bP(\ker A) \not\in Q} \overset{\ker}{\lra} \bP^3 - Q
\]
with fibre above $[K] \in \bP^3 - Q$ the Stiefel manifold
\begin{equation}\label{equation:p1p1-disc}
    \set{A \in \Hom(\bC^4, \bC^4 \oplus V)}{\ker A = K} \simeq \Hom_\catJ(\bC^4/K, \bC^4 \oplus V).
\end{equation}
Hence, Miller's splitting combined with \cref{lemma:parameterised-colim-derivatives} tells us that the derivatives of the functor
\[
    V \longmapsto \Sigma^\infty_+\set{A \in \Hom(\bC^4, \bC^4 \oplus V)}{\dim \ker A = 1 \text{ and } \bP(\ker A) \not\in Q}
\]
are given by (the shift by $\bS^K$ comes from \cref{lemma:shifted-tower})
\[
    \Theta_j = \colim_{[K] \in \bP^3 - Q} \bS^{\ad_j} \otimes \bS^{jK} \otimes \Sigma^\infty_+\Hom_\catJ(\bC^j, \bC^4/K).
\]
Likewise $S(\nu)$ fibres over $\bP^3 - Q$ with fibre above $[K] \in \bP^3 - Q$
\begin{equation}\label{equation:p1p1-sphere}
    \set{(A \in \Hom(\bC^4, \bC^4 \oplus V), \varphi \colon \ker A \to (\image A)^\perp)}{\ker A = K} \simeq \Hom_\catJ(\bC^4, \bC^4 \oplus V),
\end{equation}
where we have used the formula for the normal bundle recalled in \cref{lemma:normal-bundle-rank-matrices}. Once again, Miller's splitting gives us the desired Weiss derivatives. To summarise, the $j$th Weiss derivative of the functor~\eqref{equation:functor-sections-p1p1} is the pushout of spectra
\begin{equation}\label{equation:p1p1-pushout}
\begin{tikzcd}
{\colim\limits_{[K] \in \bP^3 - Q} \bS^{\ad_j} \otimes \Sigma^\infty_+\Hom_\catJ(\bC^j, \bC^4)} \arrow[d] \arrow[r] & {\colim\limits_{[K] \in \bP^3 - Q} \bS^{\ad_j} \otimes \bS^{jK} \otimes \Sigma^\infty_+\Hom_\catJ(\bC^j, \bC^4/K)} \\
{\bS^{\ad_j} \otimes \Sigma^\infty_+\Hom_\catJ(\bC^j, \bC^4)}  & 
\end{tikzcd}
\end{equation}
Observe in particular that for $j \geq 5$ all the spectra in the diagram vanish. Hence we see that our functor is polynomial of degree $4$.
To understand the top morphism, we need to inspect the proof Miller's splitting in \cite{Miller}. The consequence is the following.
\begin{lemma}
The morphism
\begin{equation}\label{equation:p1p1-collapse-derivative}
    \colim\limits_{[K] \in \bP^3 - Q} \bS^{\ad_j} \otimes \Sigma^\infty_+\Hom_\catJ(\bC^j, \bC^4) \lra \colim\limits_{[K] \in \bP^3 - Q} \bS^{\ad_j} \otimes \bS^{jK} \otimes \Sigma^\infty_+\Hom_\catJ(\bC^j, \bC^4/K)
\end{equation}
is induced from the Thom collapse map
\[
    \Hom_\catJ(\bC^j, \bC^4)_+ \lra \Hom_\catJ(\bC^j, \bC^4/K)_+ \wedge S^{jK}
\]
associated to the inclusion
\[
    \Hom_\catJ(\bC^j, \bC^4/K) \hookrightarrow \Hom_\catJ(\bC^j, \bC^4),
\]
where we use the identification $\bC^4/K \cong K^\perp \subset \bC^4$.
\end{lemma}
\begin{proof}[Sketch of a proof]
In \cite[Section~2]{Miller}, the following filtration is defined
\[
    F^j\Hom_\catJ(\bC^4, \bC^4 \oplus V) = \set{(\varphi \colon \bC^4 \to \bC^4, \psi \colon \bC^4 \to V)}{\dim \ker (\varphi - \mathrm{id}) \geq 4-j},
\]
and Miller shows that $F^j - F^{j-1}$ is homeomorphic to the total space of the bundle $\ad_j \oplus \Hom(\bC^j,V)$ on $\Grass(j,\bC^4)$. We have likewise a filtration $R^\bullet$ on $\Hom_\catJ(\bC^4/K, \bC^4/K \oplus K \oplus V)$ with
\[
    R^j = \set{(\varphi \colon \bC^4/K \to \bC^4/K, \psi \colon \bC^4/K \to K \oplus V)}{\dim \ker (\varphi - \mathrm{id}) \geq 3-j},
\]
and such that $R^j-R^{j-1}$ is homeomorphic to the total space of $\ad_j \oplus \Hom(\bC^j,K \oplus V)$ on $\Grass(j,\bC^4/K)$. Under the identifications~\eqref{equation:p1p1-disc} and~\eqref{equation:p1p1-sphere}, the morphism $S(\nu) \to D(\nu)$ induces a map which preserves the filtrations. Hence, Miller's splitting asserts that the morphism between the $j$th derivatives is the suspension of the natural map $F^j / F^{j-1} \to R^j / R^{j-1}$. Under the equivalences
\[
    F^j / F^{j-1} \simeq (S^{\ad_j} \wedge \Hom_\catJ(\bC^j, \bC^4)_+ \wedge S^{jV})_{hU(j)}
\]
and
\[
    R^j / R^{j-1} \simeq (S^{\ad_j} \wedge \Hom_\catJ(\bC^j, \bC^4/K)_+ \wedge S^{j(K \oplus V)})_{hU(j)}
\]
noticed by Arone \cite{Arone}, this map is induced by the Thom collapse as desired.
\end{proof}
As a consequence\footnote{If $Z \subset Y$ is a closed smooth submanifold with normal bundle $\mu$, then there is a cofibre sequence of pointed spaces $(Y-Z)_+ \to Y_+ \to \mathrm{Th}(\nu)$.}, the fibre of the morphism~\eqref{equation:p1p1-collapse-derivative} is the spectrum
\[
    \colim\limits_{[K] \in \bP^3 - Q} \bS^{\ad_j} \otimes \Sigma^\infty_+\big(\Hom_\catJ(\bC^j, \bC^4) - \Hom_\catJ(\bC^j, \bC^4/K)\big).
\]
Putting this back into the diagram~\eqref{equation:p1p1-pushout}, we conclude that the $j$th derivative $\Theta_j$ of our functor~\eqref{equation:functor-sections-p1p1} sits in the cofibre sequence
\[
    \colim\limits_{[K] \in \bP^3 - Q} \bS^{\ad_j} \otimes \Sigma^\infty_+\big(\Hom_\catJ(\bC^j, \bC^4) - \Hom_\catJ(\bC^j, \bC^4/K)\big) \to \bS^{\ad_j} \otimes \Sigma^\infty_+\Hom_\catJ(\bC^j, \bC^4) \to \Theta_j.
\]
By inspection of our constructions, the first morphism is the natural morphism induced by the subspace inclusion $\Hom_\catJ(\bC^j, \bC^4) - \Hom_\catJ(\bC^j, \bC^4/K) \hookrightarrow \Hom_\catJ(\bC^j, \bC^4)$.

\begin{example}
For $j = 1$, the subspace inclusion is equivalent to $S(\bC^4/K) \hookrightarrow S(\bC^4) \simeq S(\bC^4/K) \ast S(K)$, hence the complement is homotopy equivalent to $S(K)$. Furthermore, 
\[
    \colim\limits_{[K] \in \bP^3 - Q} S(K) \simeq S(\cO_{\bP^3 - Q}(1))
\]
and $\Hom_\catJ(\bC^1, \bC^4) \cong S(\cO_{\bP^3}(1))$ (see the previous \cref{subsection:pn-case}). Hence
\[
    \Theta_1 \simeq \bS^{\ad_1} \otimes \frac{\Sigma^\infty_+ S(\cO_{\bP^3 - Q}(1))}{\Sigma^\infty_+ S(\cO_{\bP^3}(1))} \simeq \Sigma^1\Sigma^\infty\frac{S(\cO_{\bP^3 - Q}(1))}{S(\cO_{\bP^3}(1))}.
\]
This can be further simplified by noticing that the cofibre of the inclusion $S(\cO_{\bP^3 - Q}(1)) \subset S(\cO_{\bP^3}(1))$ is the Thom space $S(\cO_{Q}(1))^{\nu(Q \hookrightarrow \bP^3)}$ where $\nu(Q \hookrightarrow \bP^3)$ is the normal bundle of the inclusion (implicitly pulled back to the sphere bundle). Now, as in the previous section, we have
\[
    \nu(Q \hookrightarrow \bP^3) = T\bP^3 - TQ = \cO(1)^{\oplus 4} - \underline{\bC} - TQ = \underline{\bC}^3 - TQ
\]
as virtual bundles on $S(\cO_{Q}(1))$. Thus
\[
    \Theta_1 \simeq \Sigma^7 S(\cO_Q(1))^{-TQ} = \Sigma^7 S(\cO_{\bP^1 \times \bP^1}(1))^{-T(\bP^1 \times \bP^1)}
\]
which indeed agrees with the answer given in \cref{maintheorem:holomorphic}.
\end{example}

\begin{example}
For $j = 4$, $\Hom_\catJ(\bC^j, \bC^4/K) = \emptyset$ and $\Hom_\catJ(\bC^j,\bC^4) = U(4)$ hence
\[
    \Theta_4 \simeq \bS^{\ad_4} \otimes \Sigma^\infty_+ U(4) \otimes \Sigma^\infty \Sigma^\mathrm{un}(\bP^3 - Q),
\]
where $\Sigma^\mathrm{un}$ is the unreduced suspension.
\end{example}

\subsection{Top derivative}\label{subsection:top-derivative}

For any variety $X$ and $\alpha \in H^2(X;\bZ)$ as in our main \cref{maintheorem:holomorphic}, we have shown that the suspension spectrum of the functor of holomorphic maps of degree $\alpha$ is $N$-polynomial, where $N = \dim H^0(X,\cL)$ and $\cL$ is any very ample holomorphic line bundle with first Chern class $\alpha$. Surprisingly, it turns out that the top derivative $\Theta_N$ admits a nice formula, as we shall explain in this section. As in the previous computations above, we will work with the unitary functor
\begin{equation}\label{equation:sections-functor-topderivatives}
    V \longmapsto \Sigma^\infty_+ \Gammahol(\cL \otimes (\bC^N \oplus V) - 0)
\end{equation}
and explain at the end how to implement the little modifications to get to the functor of holomorphic maps. We will reuse the ideas of \cref{section:polynomiality}, so we recall some notation. We had the identification~\eqref{equation:correspondence-sections-linear-maps}
\begin{equation}\label{equation:topderivative-linearmaps}
    \Gammahol(\cL \otimes (\bC^N \oplus V) - 0) \cong \set{A \in \Hom(\bC^N, \bC^N \oplus V)}{\bP(\ker A) \cap \iota X = \emptyset}
\end{equation}
where $\iota \colon X \hookrightarrow \bP(\bC^N)$ is the canonical embedding~\eqref{equation:ample-embedding}. Our first observation is:
\begin{lemma}
The stratification by rank on 
\[
    \set{A \in \Hom(\bC^N, \bC^N \oplus V)}{\bP(\ker A) \cap \iota X = \emptyset}
\]
is a Whitney stratification.
\end{lemma}
\begin{proof}
The rank stratification on $\Hom(\bC^N, \bC^N \oplus V)$ is a Whitney stratification, because\footnote{Given an orbit, the subset of points not satisfying Whitney's condition (b) form a $\mathrm{GL}_n(\bC)$-invariant subset of strictly smaller dimension by \cite[Lemma~2.4]{NguyenTrivediTrotman}. Hence it has to be empty. See \cite{MOWhitney}.} it is the stratification by orbits under the action of the algebraic group $\mathrm{GL}_n(\bC)$. The condition on the kernel is an open condition, hence the restriction of the stratification is a Whitney stratification.
\end{proof}
As a consequence, by the control theory of Mather \cite{Mather}, there exist tubular neighbourhoods $T_r$ of the strata of $A$ where $\dim \ker A = r$ such that the space~\eqref{equation:topderivative-linearmaps} is the geometric realisation (homotopy colimit) of the \v{C}ech diagram:
\begin{equation}\label{equation:cech-diagram}
\begin{tikzcd}
\bigsqcup_i T_i & {\bigsqcup_{i<j} T_i \cap T_j} \arrow[l, shift left] \arrow[l, shift right] & {\bigsqcup_{i<j<k} T_i \cap T_j \cap T_k} \arrow[l, shift left=2] \arrow[l, shift right=2] \arrow[l] & \ldots \arrow[l, shift right=4] \arrow[l, shift right=2] \arrow[l] \arrow[l, shift left=2]
\end{tikzcd}
\end{equation}

\begin{example}
Let $X = \bP^1$ and $\cL = \cO(2)$. Then the stratification has only two pieces
\[
    \set{A}{\ker A = 0} \sqcup \set{A}{\bP(\ker A) \in \bP^2 - \iota\bP^1},
\]
respectively homotopy equivalent to $T_0$ and $T_1$. Denoting by $\nu$ the normal bundle of the second stratum inside the total space, the \v{C}ech diagram is homotopy equivalent to $T_0 \leftarrow S(\nu) \to T_1$.
\end{example}

Notice that the spaces in the \v{C}ech diagram already appear in the inductive argument of the proof of \cref{theorem:holomorphic-polynomiality}. In essence, we are just condensing that argument in a single diagram. Indeed, by inductively using \cref{lemma:normal-bundle-rank-matrices} we obtain homotopy equivalences
\begin{equation}\label{equation:tubular-neighbourhoods-intersection}
    T_{r_1} \cap \cdots \cap T_{r_i} \simeq \set{(\varphi_1,\dotsc,\varphi_{i-1},A)}{\dim \ker A = r_i, \ \dim\ker \varphi_j = r_j}
\end{equation}
for $r_1 < \cdots < r_i$, and where we simply write $A$ for a generic element in~\eqref{equation:topderivative-linearmaps}, and the extra morphisms are linear maps
\begin{align*}
    \varphi_{i-1} &\colon \ker A \lra (\image A)^\perp  \\
    \varphi_j &\colon \ker \varphi_{j+1} \to \image\left(A + \sum_{k=j+1}^{i-1}\varphi_k\right)^\perp \quad \text{ for } j=1,\dotsc,i-2
\end{align*}
To make sense of the sum of maps and the orthogonal complement, we use the inner product on $\bC^N \oplus V$ to identify quotients and orthogonal complements, and notice the decomposition
\[
    \bC^N = (\bC^N/\ker A) \oplus (\ker A / \ker \varphi_{i-1}) \oplus (\ker \varphi_{i-1} / \ker \varphi_{i-2}) \oplus \cdots \oplus (\ker \varphi_2 / \ker \varphi_1) \oplus \ker \varphi_1.
\]
The map
\begin{equation}\label{equation:all-kernels}
    (\varphi_1,\dotsc,\varphi_{i-1},A) \longmapsto (\ker \varphi_1, \dotsc, \ker \varphi_{i-1},\ker A)
\end{equation}
shows that the right hand side of~\eqref{equation:tubular-neighbourhoods-intersection} bundles over the space $\Flag(r_1,\dotsc,r_i,\bC^N;\iota X)$ defined to be
\begin{equation}\label{equation:iterated-tautological-bundle}
    \set{(K_1,\dotsc,K_{i-1},K_i)}{K_i \in \Grass(r_i,\bC^N;\iota X), \ K_j \in \Grass(r_j, K_{j+1}) \text{ for } j=1,\dotsc,i-1}.
\end{equation}
We have written $\Grass(-,-)$ for the Grassmannian and
\[
    \Grass(r,\bC^N;\iota X) =  \set{P \in \Grass(r, \bC^N)}{P \cap \iota X = \emptyset} \subset \Grass(r, \bC^N)
\]
for the locus of $r$-planes not intersecting $\iota X$. The topology on the space~\eqref{equation:iterated-tautological-bundle} is the natural one obtained by iterating the tautological bundle construction on Grassmannians. The fibre of the map~\eqref{equation:all-kernels} above $(K_1,\dotsc,K_{i-1},K_i)$ is the space
\[
    \set{(\varphi_1,\dotsc,\varphi_{i-1},A)}{\ker A = K_i, \ \ker \varphi_j = K_j} \simeq \Hom_\catJ(\bC^N/K, \bC^N \oplus V),
\]
where $K = \ker(A + \varphi_1 + \cdots + \varphi_{i-1})$ and the homotopy equivalence is via adding the morphisms. By Miller's stable splitting, the functor
\[
    V \longmapsto \Sigma^\infty_+\Hom_\catJ(\bC^N/K, \bC^N \oplus V)
\]
is $(N-\dim K)$-polynomial, with $j$th derivative given by $\bS^{\ad_j} \otimes \Sigma^\infty_+\Hom_\catJ(\bC^j,\bC^N/K) \otimes \bS^{jK}$. In particular, it has a non-zero $N$th derivative if and only if $K = 0$, i.e. $r_1 = 0$. Invoking \cref{lemma:parameterised-colim-derivatives}, we obtain
\[
    \Theta_N\big( \Sigma^\infty_+(T_{r_1} \cap \cdots \cap T_{r_i})(V) \big) \simeq \begin{cases}
        0 &\text{ if } r_1 \neq 0 \\
        \bS^{\ad_N} \otimes \Sigma^\infty_+U(N) \otimes \Sigma^\infty_+\Flag(r_1,\dotsc,r_i,\bC^N;\iota X) &\text{ if } r_1 = 0.
    \end{cases}
\]
The top derivative $\Theta_N$ of our functor of interest~\eqref{equation:sections-functor-topderivatives} is then obtained by combining this computation and the colimit diagram~\eqref{equation:cech-diagram}. This yields:
\[
\begin{tikzcd}
    \Theta_N \simeq \colim \bigg( \Theta_N \Sigma^\infty_+T_0 & {\bigsqcup_{0<i} \Theta_N\Sigma^\infty_+(T_0 \cap T_i)} \arrow[l] & {\bigsqcup_{0<i<j} \Theta_N \Sigma^\infty_+(T_0 \cap T_i \cap T_j)} \arrow[l, shift left=1] \arrow[l, shift right=1] \ \cdots \bigg)
\end{tikzcd}
\]
By inspection of our constructions, the morphisms between the derivatives are all induced by the natural projections and inclusions between the flag varieties $\Flag(r_1,\dotsc,r_i,\bC^N;\iota X)$. This was already observed in an easier case in the previous \cref{subsection:p1p1-case}.
We therefore obtain
\[
\begin{tikzcd}[column sep=1em]
    \Theta_N \simeq \bS^{\ad_N} \otimes \Sigma^\infty_+U(N) \otimes \colim \bigg( \bS^0 & {\bigsqcup_{0<i} \Sigma^\infty_+\Flag(0,i,\bC^N;\iota X)} \arrow[l] & {\bigsqcup_{0<i<j} \Sigma^\infty_+\Flag(0,i,j,\bC^N;\iota X)} \arrow[l, shift left=1] \arrow[l, shift right=1] \ \cdots \bigg)
\end{tikzcd}
\]
Recall that for a space $Y$, the cofibre (in pointed spaces) of $Y_+ \to S^0$ is the unreduced suspension $\Sigma^\mathrm{un}Y$. Applied to our situation, we get
\[
\begin{tikzcd}[column sep=1em]
    \Theta_N \simeq \bS^{\ad_N} \otimes \Sigma^\infty_+U(N) \otimes \Sigma^\infty\Sigma^\mathrm{un}\colim \bigg( {\bigsqcup_{0<i} \Flag(0,i,\bC^N;\iota X)} & {\bigsqcup_{0<i<j} \Flag(0,i,j,\bC^N;\iota X)} \arrow[l, shift left=1] \arrow[l, shift right=1] \ \cdots \bigg)
\end{tikzcd}
\]
where the colimit is taken in unpointed spaces. Now observe that this colimit is exactly the geometric realisation of the poset of non-empty linear subspaces $\emptyset \neq P \subset \bP(\bC^N) - \iota X$ ordered by inclusion. See e.g. \cite[Section~3]{Zivaljevic} for basics on topological posets. Let us record our computations:
\begin{proposition}\label{proposition:top-derivative}
Let $X$ be a smooth projective complex variety, and $\cL$ be a very ample line bundle on $X$. Let $N = \dim H^0(X,\cL)$, and $\iota \colon X \hookrightarrow \bP(\bC^N)$ be the canonical embedding. The $N$th derivative of the unitary functor
\[
    V \longmapsto \Sigma^\infty_+ \Gammahol(\cL \otimes (\bC^N \oplus V) - 0)
\]
is given by the formula
\[
    \Theta_N \simeq \bS^{\ad_N} \otimes \Sigma^\infty_+U(N) \otimes \Sigma^\infty\Sigma^\mathrm{un}|\Grass(X,\cL)|
\]
where $|\Grass(X,\cL)|$ is the geometric realisation of the topological poset $\Grass(X,\cL)$ of non-empty linear subspaces $\emptyset \neq P \subset \bP(\bC^N) - \iota X$ ordered by inclusion. The action of the unitary group $U(N)$ on $\Theta_N$ is the natural action on $\bS^{\ad_N} \otimes \Sigma^\infty_+U(N)$.
\end{proposition}

\begin{example}
For $X = \bP^1$ and $\cL = \cO(3)$, we have
\begin{align*}
    \Flag(0,\bC^4;\iota \bP^1) &= \{0\}, \\
    \Flag(0,1,\bC^4;\iota \bP^1) &\cong \bP^3 - \iota \bP^1, \\
    \Flag(0,2,\bC^4;\iota \bP^1) &\cong \Grass(2,\bC^4;\iota \bP^1), \\
    \Flag(0,1,2,\bC^4;\iota \bP^1) &\cong \set{(P,x)}{P \in \Grass(2,\bC^4;\iota \bP^1), \ x \in \bP(P)} \text{ (the tautological bundle)}.
\end{align*}
The two non-trivial natural maps appearing in the \v{C}ech diagram are the projections
\[
    \Flag(0,1,2,\bC^4;\iota \bP^1) \lra \Flag(0,2,\bC^4;\iota \bP^1), \quad (P,x) \longmapsto P
\]
and
\[
    \Flag(0,1,2,\bC^4;\iota \bP^1) \lra \Flag(0,1,\bC^4;\iota \bP^1), \quad (P,x) \longmapsto x \in \bP(P) \subset \bP^3 - \iota \bP^1.
\]
The geometric realisation of the poset $\Grass(\bP^1,\cO(3))$ is thus homotopy equivalent to the coequaliser of the diagram
\[
\begin{tikzcd}
    \set{(P,x)}{P \in \Grass(2,\bC^4;\iota \bP^1), \ x \in \bP(P)} \times \Delta^1 \arrow[r, shift right=1] \arrow[r, shift left=1] & (\bP^3 - \iota \bP^1) \sqcup \Grass(2,\bC^4;\iota \bP^1) 
\end{tikzcd}
\]
where the maps are the ones given above.
\end{example}

Let us now explain, as promised at the beginning of this section, how to obtain the top derivative of the functor
\begin{equation}\label{equation:topderivative-mapping-space}
    V \longmapsto \Sigma^\infty_+\Hol_\alpha(X, \bP(\bC^N \oplus V)),
\end{equation}
where $\alpha = c_1(\cL)$ as chosen at the beginning. When $H^1(X;\bZ) = 0$, then
\[
    \Hol_\alpha(X, \bP(\bC^N \oplus V)) \cong \Gammahol(\cL \otimes (\bC^N \oplus V) - 0) / \bC^\times
\]
and \cref{subsubsection:canonical-u1-action} explains in that case that the $N$th derivative of~\eqref{equation:topderivative-mapping-space} is given by 
\[
    \Theta_N \simeq \bS^{\ad_N} \otimes \Sigma^\infty_+PU(N) \otimes \Sigma^\infty\Sigma^\mathrm{un}|\Grass(X,\cL)|
\]
where $PU(N)$ is the projective unitary group. This applies in particular in the case of $X = \bP^n$. Notice especially that for $\bP^1$ we have computed \emph{all} the Weiss derivatives, by the combination of \cref{maintheorem:holomorphic} and \cref{proposition:top-derivative}. In general, we need to proceed as explained in \cref{subsubsection:picard-parameter}. That is, we pick a Poincaré line bundle $\cP$ and consider the canonical embedding
\[
    \epsilon \colon \Pic^\alpha(X) \times X \longhookrightarrow \bP\big( (p_*\cP)^\vee \big)
\]
where $p \colon \Pic^\alpha(X) \times X \to \Pic^\alpha(X)$ is the first projection. We can now repeat the whole argument with the flag varieties replaced by 
\[
    \Flag(r_1,\dotsc,r_i,\alpha; \epsilon) = \left\{
    ([\cL], K_1, \dotsc, K_i)\ 
    \middle\vert 
    \begin{array}{l}
        [\cL] \in \Pic^\alpha(X), \ \dim K_j = r_j, \text{ and } \\
        K_1 \subset K_2 \subset \cdots \subset K_i \subset H^0(X, \cL)^\vee \\
        \text{such that } K_i \cap \epsilon(\{[\cL]\} \times X) = \emptyset 
    \end{array} \right\}.
\]
In other words, we simply let $\cL$ vary in the moduli $\Pic^\alpha(X)$. We leave the details to the interested reader.

\printbibliography

\end{document}